\documentclass[a4paper]{amsart}

\usepackage{amscd,amsthm,amsfonts,amssymb,amsmath}
\usepackage{fullpage}
\usepackage[all]{xy}
\usepackage{colonequals}

    \newcommand{\BC}{{\mathbb {C}}} 
     \newcommand{\BF}{{\mathbb {F}}}

    \newcommand{\BQ}{{\mathbb {Q}}} \newcommand{\BR}{{\mathbb {R}}}
     \newcommand{\BT}{{\mathbb {T}}}

     \newcommand{\BZ}{{\mathbb {Z}}}

    \newcommand{\CC}{{\mathcal {C}}} 
    \newcommand{\CE}{{\mathcal {E}}} 
     \newcommand{\CH}{{\mathcal {H}}}

    \newcommand{\CM}{{\mathcal {M}}} 
     \newcommand{\CP}{{\mathcal {P}}}
     \newcommand{\CR}{{\mathcal {R}}}
     \newcommand{\CT}{{\mathcal {T}}}

     \newcommand{\ff}{{\mathfrak{f}}}

    \newcommand{\fm}{{\mathfrak{m}}} \newcommand{\fn}{{\mathfrak{n}}}
     
    \newcommand{\fq}{{\mathfrak{q}}}

     \newcommand{\fX}{{\mathfrak{X}}}

    \newcommand{\Ann}{{\mathrm{Ann}}}

    \newcommand{\codim}{{\mathrm{codim}}}

    \newcommand{\Div}{{\mathrm{Div}}} 
    \newcommand{\End}{{\mathrm{End}}}

    \newcommand{\Frob}{{\mathrm{Frob}}}

    \newcommand{\Gal}{{\mathrm{Gal}}} \newcommand{\GL}{{\mathrm{GL}}}

    \newcommand{\id}{{\mathrm{id}}}\renewcommand{\Im}{{\mathrm{Im}}}

    \newcommand{\Res}{{\mathrm{Res}}}

    \newcommand{\SL}{{\mathrm{SL}}}

    \newcommand{\tor}{{\mathrm{tor}}}

    \theoremstyle{plain}
    \newtheorem{thm}{Theorem}[section] \newtheorem{cor}[thm]{Corollary}
    \newtheorem{lem}[thm]{Lemma}  \newtheorem{prop}[thm]{Proposition}
     \newtheorem{defn}[thm]{Definition}

\theoremstyle{remark} \newtheorem{remark}[thm]{Remark}
\theoremstyle{remark} 
\theoremstyle{remark} 

    \newcommand{\etale}{\'{e}tale~}
    \newcommand{\et}{\'{e}t}
    \newcommand{\Poincare}{Poincar\'{e}~}\renewcommand{\et}{{\text{\'{e}t}}}

    \numberwithin{equation}{section}

\begin{document}

\title{Quadratic torsion subgroups of modular Jacobian varieties}
\author{Yuan Ren}

\address{School of Mathematical Science, Sichuan Normal University, Chengdu, Sichuan, China}
\email{rytgyx@126.com}

\begin{abstract}Let $D$ be an odd square-free positive integer and $C$ a divisor of $D$. For any quadratic character $\chi$ modulo $C$, we prove that the $\chi$-part of the group $J_0(DC)_\text{tor}$ of torsion points of $J_0(DC)$ coincides with the $\chi$-part of its cuspidal subgroup, away from those primes of bad reduction or where possible congruences between oldforms and newforms occur.
\end{abstract}

\maketitle

\section{Introduction}
For a positive integer $N$, let $X_0(N)$ be the modular curve of level $N$ over $\BQ$. Let $J_0(N)$ be the Jacobian variety of $X_0(N)$. There has been a lot of work concerning the torsion part $J_0(N)(\BQ)_{\tor}$ for the group $J_0(N)(\BQ)$ of $\BQ$-rational points of $J_0(N)$. Let $\CC_0(N)$ be the cuspidal subgroup of $J_0(N)$ generated by those degree-zero divisor classes which are supported at the cusps, and let $\CC_0(N)(\BQ)=\CC_0(N)^{\Gal(\overline{\BQ}/\BQ)}$ be its $\BQ$-rational subgroup. Then we know that
\begin{itemize}
  \item $J_0(p)(\BQ)_\tor=\CC_0(p)(\BQ)$ for any prime $p$ (See \cite{M});
  \item $J_0(p^r)(\BQ)_\tor\otimes_{\BZ}\BZ[1/6p]=\CC_{p^r}(\BQ)\otimes_{\BZ}\BZ[1/6p]$ for any prime $p\geq5$ and any integer $r\geq1$ (See \cite{Lo} or \cite{L});
  \item $J_0(D)(\BQ)_\tor\otimes_{\BZ}\BZ[1/6]=\CC_0(D)\otimes_{\BZ}\BZ[1/6]$ for any square-free positive integer $D$ (See \cite{Oh}). Note that, when $D$ is square-free, all the cusps of $X_0(D)$ are $\BQ$-rational and hence $\CC_0(D)=\CC_0(D)(\BQ)$.
\end{itemize}
However, since cusps are  not $\BQ$-rational in general, it is natural to ask: What is the role played by $C_0(N)$ in the group $J_0(N)_\tor$ of all torsion points of $J_0(N)$? In this paper we take a first step in investigating this question. Let $N=DC$ with $D$ an odd square-free positive integer and $C$ dividing $D$. Then all cusps of $X_0(DC)$ are $\BQ(\mu_C)$-rational (See \S4.1). So, for any \emph{quadratic} Dirichlet character $\chi$ modulo $C$, we can define
\begin{align*}
  \CC_0(DC)(\chi)&\colonequals\{P\in \CC_0(DC)|\ \sigma(P)=\chi(\sigma)\cdot P \text{ for any }\sigma\in \Gal(\overline{\BQ}/\BQ)\}\\
  J_0(DC)(\chi)&\colonequals\{P\in J_0(DC)|\ \sigma(P)=\chi(\sigma)\cdot P \text{ for any }\sigma\in \Gal(\overline{\BQ}/\BQ)\}.
\end{align*}
The main result of this paper is the following theorem.

\begin{thm}\label{M} Let $D$ be an odd square-free positive integer and $C$ a divisor of $D$. If $\chi$ is a quadratic character modulo $C$, then
\[J_0(DC)(\chi)[q^\infty]=\CC_0(DC)(\chi)[q^\infty]\]
for any prime $q$ satisfying $(q,6D\cdot\varphi(D/C)\cdot\varpi(C/\ff_\chi))=1$, where $\ff_\chi$ is the conductor of $\chi$, $\varphi(D/C)=\prod_{p\mid{D/C}}(p-1)$ and $\varpi(C/\ff_\chi)=\prod_{p\mid{C/\ff_\chi}}(p^2-1)$.
\end{thm}

Here is a sketch for the proof of the theorem. Let $\BT=\BZ[\{T_\ell\}_{\ell}]$ be the full Hecke algebra of level $DC$, where $\ell$ runs over all primes. Let $I_\chi=(\{T_\ell-\chi(\ell)-\chi(\ell)\ell\}_{\ell\nmid D})$. Then we can show that $J_0(DC)(\chi)[q^\infty]$ is a module over the finite ring $\BT/I_\chi$. So it suffices for us to prove $J_0(DC)(\chi)[\fm^\infty]=C_0(DC)(\chi)[\fm^\infty]$ for every maximal ideal $\fm$ of $\BT$ containing $I_{\chi}$. Our proof of this equality consists of the following three ingredients:

\begin{enumerate}
  \item For any maximal ideal $\fm$ as above we can associate an Eisenstein series $E$, which produces a subgroup $\CC(E)$ of $C_0(DC)(\chi)[\fm^\infty]$ by the method of Stevens;
  \item On the other hand we obtain an Eisenstein ideal $I(E)$ from the annihilators of $E$ in the Hecke algebra, whose index in $\BT$ agrees with the order of $\CC(E)$ (at least away from $6D$);
  \item Finally, under our assumption $(q,6D\cdot\varphi(D/C)\cdot\varpi(C/\ff_\chi))=1$, we will prove that $J_0(DC)(\chi)[\fm^\infty]$ is a cyclic $\BT/I(E)$-module, which implies $|J_0(DC)(\chi)[\fm^\infty]|\leq|\CC(E)|$. Since $\CC(E)$ is contained in $J_0(DC)(\chi)[\fm^\infty]$, the theorem follows. Note that we make the assumption about $q$ to deal with the possible congruence between $q$-oldforms and $q$-newforms.
\end{enumerate}

The paper is organized as follows. After recalling some preliminaries in \S2, we construct in \S3 an eigen-basis for the space of Eisenstein series of weight two and level $DC$. While all these Eisenstein series are interesting, we will in this paper focus on those with quadratic characters. For any such an Eisenstein series, we calculate in \S4 the order of its associated cuspidal subgroup. Then we give the proof of Theorem~\ref{M} in the final section.
\\

\textbf{Notations}:

For any abelian group $A$, denote $A_p\colonequals A\otimes_{\BZ}\BZ_p$.

For any positive integer $N=\prod_{p\mid N}p^{v_p(N)}$, let $\nu(N)\colonequals\sum_{p\mid N}v_p(N)$ and $\psi(N)\colonequals\prod_{p\mid N}(p+1)$.

$\CH=\{z\in\BC|\Im(z)>0\}$ is the \Poincare upper half-plane. Let $\fq:\CH\rightarrow\BC,z\mapsto e^{2\pi iz},$ be the function on $\CH$ which will is used in the Fourier expansions of modular forms.

For any function $g$ on the upper half-plane and any $\gamma=\left(
                        \begin{array}{cc}
                          a & b \\
                          c & d \\
                        \end{array}
                      \right)
\in\GL^+_2(\BR)$, we denote by $g|\gamma$ to be the function $z\mapsto\det(\gamma)(cz+d)^{-2}g(\gamma z)$, where $\gamma z=\frac{az+b}{cz+d}$.

If $g$ is a modular form of some level and $x$ is a cusp, then we denote $a_0(g;[x])$ to be the constant term of the Fourier expansion of $g$ at $x$.

\section{Preliminaries}
For any positive integer $N$, the modular curve $X_0(N)$ is the coarse moduli space classifying all pairs $(E,H)$, where $E$ is a (generalized) elliptic curve over some $\BQ$-scheme, and $H$ is a cyclic subgroup of $E$ of order $N$. This is a smooth projective curve over $\BQ$. We denote by $J_0(N)$ the Jacobian variety of $X_0(N)$. The Atkin-Lehner operator $w_N$ is the automorphism on $X_0(N)$ defined by
\begin{align*}
  w_N(E,H)=(E/H,E[N]/H).
\end{align*}
We have $w^2=\id$. Denote also by $w_N$ the induced automorphism on $J_0(N)$. For any prime $\ell$, let $X_0(N,\ell)$ be the coarse moduli space classifying all triples $(E,H,G)$, where $(E,H)$ is as above and $G$ a cyclic subgroup of order $\ell$ such that $G\bigcap H=\emptyset$. Then there are two degeneracy maps
\begin{align*}
  \begin{cases}
    \alpha_{\ell}:~X_0(N,\ell)\rightarrow X_0(N)\\
    \beta_{\ell}:~X_0(N,\ell)\rightarrow X_0(N),
  \end{cases}
\end{align*}
where $\alpha_{\ell}(E,H,G)\colonequals(E,H)$ and $\beta_\ell(E,H,G)\colonequals(E/G,(H+G)/G)$. For any prime $\ell$, we define an endomorphism on $J_0(N)$ as
\begin{align*}
    T^{(N)}_\ell\colonequals\beta_{\ell*}\circ\alpha^*_\ell,
\end{align*}
and define $\BT(N)$ to be the sub-$\BZ$-algebra of $\End_{\BQ}(J_0(N))$ generated by $T^{(N)}_\ell$ for all primes $\ell$. Unless necessary we will denote $T^{(N)}_\ell$ simply as $T_\ell$.

Let $\CM_2(\Gamma_0(N),\BC)$ be the space of weight-two modular forms of level $\Gamma_0(N)$. Let $S_2(\Gamma_0(N),\BC)$ (resp. $\CE_2(\Gamma_0(N),\BC)$) be the subspace of cusp forms (resp. Eisenstein series) of $\CM_2(\Gamma,\BC)$, so that we have
\[\CM_2(\Gamma_0(N),\BC)=S_2(\Gamma_0(N),\BC)\oplus\CE_2(\Gamma_0(N),\BC).\]
For any prime $\ell$, we have the Hecke operator $\CT^{(N)}_\ell$ acting on $\CM_2(\Gamma_0(N),\BC)$ as
\begin{align*}
\CT^{(N)}_\ell(g)=
 \begin{cases}
   \sum^{\ell}_{k=0}g|\left(
                        \begin{array}{cc}
                          1 & k \\
                          0 & \ell \\
                        \end{array}
                      \right)+g|\left(
                                  \begin{array}{cc}
                                    \ell & 0 \\
                                    0 & 1 \\
                                  \end{array}
                                \right)& \text{if } \ell\nmid N\\
   \sum^{\ell}_{k=0}g|\left(
                        \begin{array}{cc}
                          1 & k \\
                          0 & \ell \\
                        \end{array}
                      \right)& \text{if }\ell\mid N.
 \end{cases}
\end{align*}
Define $\CT(N)$ to be the sub-$\BZ$-algebra of $\End_{\BQ}(\CM_2(\Gamma_0(N),\BC))$ generated by $\CT^{(N)}_\ell$ for all primes $\ell$. And likely we will denote $\CT^{(N)}_\ell$ as $\CT_\ell$ for simplicity. Under the natural identification
\begin{align*}
 S_2(\Gamma_0(N),\BC) \simeq H^0(J_0(N),\Omega_{J_0(N)})\otimes_{\BQ}\BC,
\end{align*}
we have $\CT_\ell$ coincides with $T_\ell$. So we can view $\BT(N)$ as the restriction of $\CT(N)$ on the space of cusp forms.

Let $cusp(\Gamma_0(N))$ be the set of cusps of $X_0(N)$. Denote by $Y_0(N)$ the complement of $cusp(\Gamma_0(N))$ in $X_0(N)$. For any form $g$ in $\CM_2(\Gamma_0(N),\BC)$, let $\omega_g$ be the differential on $X_0(N)(\BC)$ whose pullback to $\CH$ equals $g(z)dz$. If $E\in\CE_2(\Gamma_0(N),\BC)$, then $\omega_E$ is holomorphic on $Y_0(N)$ and hence there is an induced homomorphism
\begin{align*}
\xi_E:~H_1(Y_0(N)(\BC),\BZ)\rightarrow\BC,\ [c]\mapsto\int_{c}\omega_E,
\end{align*}
where $[c]$ is a homology class represented by a $1$-cycle $c$ on $Y_0(N)(\BC)$. Note that, for any small circle $c_x$ around a cusp $x$, we have
\begin{align*}
  \int_{c_x}\omega_{E}&=2\pi i\cdot \Res_x(\omega_E)\\
  &=e_x\cdot a_0(E;[x]).
\end{align*}
where $\Res_x(\omega_E)$ is the residue of $\omega_E$ at $[x]$, $e_x$ is the ramification index of $X_0(N)$ over $X(1)$ at $x$, and $a_0(E;[x])$ is the constant term of the Fourier expansion of $E$ at the cusp $x$ (\cite{St}, P36). Let $\Div^0(cusp(\Gamma_0(N));\BC)=\Div^0(cusp(\Gamma_0(N));\BZ)\otimes_\BZ\BC$. Define the following $\BC$-linear map (see \cite{St}, P35, or \cite{St2}, \S1)
\[\delta_{N}:\CE_2({\Gamma_0(N),\BC})\rightarrow \Div^0(cusp(\Gamma_0(N));\BC),\]
such that
\begin{align*}
  E\mapsto& 2\pi i\sum_{x\in{cusp(\Gamma_0(N))}}\Res_x(\omega_E)\cdot[x].
\end{align*}

\begin{defn}\label{def}For any $E$ in $\CE_2(\Gamma_0(N),\BC)$, let $\CR_{N}(E)$ be the sub-$\BZ$-module of $\BC$ generated by the coefficients of $\delta_{N}(E)$, and $\CR_N(E)^{\vee}$ be the $\BZ$-dual of $\CR_N(E)$.
\begin{enumerate}
  \item Define the cuspidal subgroup $\CC_{N}(E)$ associated with $E$ to be the subgroup of $J_0(N)$ generated by $\{w_N\left(\phi(\delta_N(E))\right)\}_{\phi\in\CR_N(E)^{\vee}}$, where $\phi(\delta_N(E))$ is the divisor obtained by applying $\phi$ to the coefficients of $\delta_N(E)$;
  \item Define the Eisenstein ideal $I_N(E)$ of $E$ to be the image of $\Ann_{\CT(N)}(E)$ in $\BT(N)$, where $\Ann_{\CT(N)}(E)$ is the annihilator of $E$ in $\CT(N)$.
  \item Define the periods $\CP_N(E)$ of $E$ to be the image of $\xi_E$.
\end{enumerate}
\end{defn}

\begin{remark}\label{rmk}
The above definition of $\CC_{N}(E)$ is slightly different from that given in \S1.8 of \cite{St} (see also \S1 of \cite{St2})by adding an action of $w_N$. Under this modification, we claim that $\CC_{N}(E)$ is annihilated by $I_N(E)$. To see this, note that on the one hand we have $\delta_N(\CT_\ell(E))={^t}\CT_\ell(\delta_N(E))$ for any prime $\ell$, where $^t\CT_\ell$ is the transpose of $\CT_\ell$ (see P110 of \cite{St}). On the other hand, we have $w_N\circ{^t\CT_\ell}=\CT_\ell\circ w_N$ by (4) on P444 of \cite{Ri2}, so the claim follows. However, since $w_N$ is an involution, this modification does not change the order of $\CC_{N}(E)$.
\end{remark}

\begin{remark}
In the same way as above, we can define a $\BC$-linear map
\begin{align*}
  \delta_{\Gamma_1(N)}:\CE_2({\Gamma_1(N),\BC})\rightarrow \Div^0(cusp(\Gamma_1(N));\BC);
\end{align*}
and, for any $E$ in $\CE_2(\Gamma_1(N),\BC)$, we have similarly a homomorphism
\begin{align*}
  \xi'_E:H_1(Y_1(N)(\BC),\BZ)\rightarrow\BC.
\end{align*}
In particular we can define $\CR_{\Gamma_1(N)}(E)$ to be the module generated by the coefficients of $\delta_{\Gamma_1(N)}(E)$ and define $\CP_{\Gamma_1(N)}(E)\colonequals\Im(\xi'_E)$.
\end{remark}

Take an $E$ in $\CE_2(\Gamma_0(N),\BC)$. Then $E$ also belongs to $\CE_2(\Gamma_1(N),\BC)$ as $\Gamma_0(N)\subseteq\Gamma_1(N)$. If the Fourier expansion of $E=\sum^{\infty}_{n=0}a_n(E;[\infty])\cdot \fq^n$, then for any Dirichlet character $\eta$, we define
\begin{align*}
L(E,\eta,s):=\sum^{\infty}_{n=1}\frac{a_n(E;[\infty])\cdot\eta(n)}{n^s}.
\end{align*}
Denote $S_N$ to be the set of all primes $p$ satisfying $p\equiv-1\pmod{4N}$. Let $\fX_N$ be the set of all non-quadratic Dirichlet characters $\eta$ whose conductor is a prime in $S_N$. And let $\fX^{\infty}_N$ be the set of all non-quadratic Dirichlet characters $\eta$ whose conductor is a power of a prime in $S_N$. For any $\eta\in\fX^{\infty}_N$ of conductor $p^d_\eta$, where $p_\eta\in S_N$ and $d$ is a positive integer, let
\begin{align*}
  &\Lambda(E,\eta,1):=\frac{\tau(\overline{\eta})\cdot L(E,\eta,1)}{2\pi i},\\
  &\Lambda_{\pm}(E,\eta,1):=\frac{1}{2}(\Lambda(E,\eta,1)\pm\Lambda(E,\eta\cdot(\frac{}{p_\eta}),1)).
\end{align*}
Here $(\frac{}{p_\eta})$ is the Legendre symbol associated to $p_\eta$. By Theorem 1.3 of \cite{St2}, if $\CM$ is a finitely generated sub-$\BZ$-module of $\BC$, then the following are equivalent:
\begin{description}
  \item[St1] $\CP_{\Gamma_1(N)}(E)\subseteq\CM$;
  \item[St2] $\CR_{\Gamma_1(N)}(E)\subseteq\CM$ and $\Lambda_{\pm}(E,\eta,1)\in\CM[\eta,\frac{1}{p_\eta}]$ for any $\eta\in\fX_{N}$;
  \item[St3] $\CR_{\Gamma_1(N)}(E)\subseteq\CM$ and $\Lambda_{\pm}(E,\eta,1)\in\CM[\eta,\frac{1}{p_\eta}]$ for any $\eta\in\fX^{\infty}_N$.
\end{description}
Let $\pi^*_N: J_0(N)\rightarrow J_1(N)$ be the induced homomorphism, where $\pi_N:X_1(N)\rightarrow X_0(N)$ is the natural projection. Let $\Sigma_{N}=\ker(\pi^*_N)$ be the Shimura subgroup of $J_0(N)$, which is finite and of multiplicative type as a group scheme over $\BQ$. Put
\[A^{(s)}_N(E):=\left(\CP_{\Gamma_1(N)}(E)+\CR_N(E)\right)/\CR_N(E).\]
Then, by (4.3) of \cite{St2}, there is an exact sequence
\[
\xymatrix@C=0.5cm{
  0 \ar[r] & \Sigma_N\bigcap \CC_{N}(E) \ar[r] & \CC_{N}(E) \ar[r] & A^{(s)}_N(E) \ar[r] & 0, }
  \]
which enables us to determine the order of $\CC_{N}(E)/\left(\Sigma_N\bigcap \CC_{N}(E)\right)$.
\\

At the end of this section, we recall some basic properties of the collection of functions $\{\phi_{\underline{x}}\}_{\underline{x}\in(\BQ/\BZ)^{\oplus2}}$ due to Hecke (see \cite{St}, Chapter 2, \S2.4) which will be needed later. For any $\underline{x}=(x_1,x_2)\in(\BQ/\BZ)^{\oplus2}$, the Fourier expansion of $\phi_{\underline{x}}$ at infinity is
\begin{align}
  \phi_{\underline{x}}(z)+\delta(\underline{x})\cdot\frac{i}{2\pi(z-\overline{z})}=\frac{1}{2} B_2(x_1)-P_{\underline{x}}(z)-P_{-\underline{x}}(z)
\end{align}
for any $z\in\CH$, where $B_2(t)=\langle{t}\rangle^2-\langle{t}\rangle+\frac{1}{6}$ is the second Bernoulli polynomial and
\begin{align}
P_{\underline{x}}(z)=\sum_{k\in\BQ_{>0},k\equiv x_1(1)}k\sum^{\infty}_{m=1}e^{2\pi im(kz+x_2)},
\end{align}
and $\delta(\underline{x})$ is defined to be $1$ or $0$ according to $\underline{x}=0$ or not. If $\underline{x}\neq0$, then $\phi_{\underline{x}}$ is a (holomorphic) Eisenstein series. Moreover, for any $\underline{x}\in(\BQ/\BZ)^{\oplus2}$ and $\gamma\in\SL_2(\BZ)$, we have
\begin{align}
  \phi_{\underline{x}}|\gamma=\phi_{\underline{x}\cdot\gamma},
\end{align}\label{distribution}
where ${\underline{x}\cdot\gamma}$ is the natural right action of $\gamma$ on the row vector of length two. The whole collection of functions satisfy the following \emph{distribution law}
\begin{align}
  \phi_{\underline{x}}=\sum_{\underline{y}:\ \underline{y}\cdot \alpha=\underline{x}}\phi_{\underline{y}}|\alpha,
\end{align}
where $\alpha$ is any matrix in $M_2(\BZ)$ with positive determinant.

\section{An eigen-basis for $\CE_2(\Gamma_0(DC),\BC)$}
\subsection{}We first introduce some operators on the space $C^{\infty}(\CH,\BC)$ of smooth $\BC$-valued functions on $\CH$. Similar constructions have already been used in Definition 4.6 of \cite{St2}. When the character is trivial, these operators have been introduced by Yoo (See Definition 2.5 of \cite{Yoo}). For any prime $p$, let
\[\gamma_p\colon C^{\infty}(\CH,\BC)\rightarrow C^{\infty}(\CH,\BC),\ g\mapsto g|\left(
                                                                                           \begin{array}{cc}
                                                                                             p & 0 \\
                                                                                             0 & 1 \\
                                                                                           \end{array}
                                                                                         \right)
.\]
If $\chi$ is a Dirichlet character of conductor $\ff_\chi$, then, for any prime $p\nmid \ff_\chi$, we define
\begin{align*}
[p]^+_\chi&\colonequals1-\chi(p)\gamma_p,\\
[p]^-_\chi&\colonequals1-p^{-1}\chi(p)^{-1}\gamma_p.
\end{align*}
If $p_1$ and $p_2$ are two primes not dividing $\ff_\chi$, then $[p_1]^+_\chi,[p_1]^-_\chi,[p_2]^+_\chi$ and $[p_2]^-_\chi$ commute with each other. So we can define, for any positive square-free integer $M$ prime to $\ff_\chi$, the following operators
\[[M]^{\pm}_\chi:=[p_1]^{\pm}_\chi\circ[p_2]^{\pm}_\chi\circ\cdots\circ[p_k]^{\pm}_\chi,\]
with $M=p_1\cdot p_2\cdot\cdot\cdot p_k$ in any order. When $\chi=1$, we write $[M]^{\pm}_1$ as $[M]^{\pm}$ for simplicity.

\begin{lem}\label{1.0}Let $N$ be a positive integer. If $\chi$ is a character of conductor $\ff_\chi$ and $p$ a prime such that $(p,\ff_\chi)=1$, then  $[p]^{\pm}_\chi M_2(\Gamma_0(N),\BC)\subseteq M_2(\Gamma_0(Np),\BC)$, $[p]^{\pm}_\chi(S_2(\Gamma_0(N),\BC))\subseteq S_2(\Gamma_0(Np),\BC)$ and $[p]^{\pm}_\chi(\CE_2(\Gamma_0(N),\BC))\subseteq \CE_2(\Gamma_0(Np),\BC)$.
\end{lem}
\begin{proof}
It suffices to prove the same assertions for the operator $\gamma_p$. The first two assertions are clear. Since $\{\phi_{\underline{x}}\}_{\underline{x}\in(\BQ/\BZ)^{\oplus2}-0}$ is a basis for the space of Eisenstein series of weight two, and $\gamma_p(\phi_{\underline{x}})=\phi_{\underline{x}}|\gamma_p$ is still an Eisenstein series by (2.4), the third assertion follows.
\end{proof}

\begin{lem}\label{1.1}
Let $N$ be a positive integer. If $\chi$ is a character of conductor $\ff_\chi$ and $p$ a prime such that $(p,\ff_\chi)=1$, then
\begin{enumerate}
  \item $\CT^{(Np)}_\ell\circ[p]^{\pm}_\chi=[p]^{\pm}_\chi\circ\CT^{(N)}_\ell$ for any prime $\ell\neq p$;
  \item If $p\nmid N$, then $\CT^{(Np)}_p\circ[p]^{+}_\chi=\CT^{(N)}_p-\gamma_p-p\cdot\chi(p)$ and $\CT^{(Np)}_p\circ[p]^{-}_\chi=\CT^{(N)}_p-\gamma_p-\chi(p)^{-1}$;
  \item If $p\mid N$, then $\CT^{(Np)}_p\circ[p]^{+}_\chi=\CT^{(N)}_p-p\cdot\chi(p)$ and $\CT^{(Np)}_p\circ[p]^{-}_\chi=\CT^{(N)}_p-\chi(p)^{-1}$.
\end{enumerate}
\end{lem}
\begin{proof}
If $(p,\ell)=1$, then $\gamma_p$ commutes with the $\ell$-th Hecke operators, so the first assertion follows. If $p\nmid N$, then we have
\begin{align*}
  \CT^{(Np)}_p\circ[p]^{+}_\chi(g)&=g|\left[1-\chi(p)\cdot\left(
                                                                    \begin{array}{cc}
                                                                      p & 0 \\
                                                                      0 & 1 \\
                                                                    \end{array}
                                                                  \right)
  \right]|\sum^{p-1}_{k=0}\left(
                               \begin{array}{cc}
                                 1 & k \\
                                 0 & p \\
                               \end{array}
                             \right)\\
  &=g|\sum^{p-1}_{k=0}\left(
                               \begin{array}{cc}
                                 1 & k \\
                                 0 & p \\
                               \end{array}
                             \right)-\chi(p)\cdot g|\sum^{p-1}_{k=0}\left(
                               \begin{array}{cc}
                                 p & pk \\
                                 0 & p \\
                               \end{array}
                             \right)\\
  &=\CT^{(N)}_p(g)-f|\gamma_p-p\cdot\chi(p)\cdot g,
\end{align*}
for any $g\in M_2(\Gamma_0(N),\BC)$, and
\begin{align*}
  \CT^{(Np)}_p\circ[p]^{-}_\chi(g)&=g|\left[1-p^{-1}\cdot\chi(p)^{-1}\cdot\left(
                                                                    \begin{array}{cc}
                                                                      p & 0 \\
                                                                      0 & 1 \\
                                                                    \end{array}
                                                                  \right)
  \right]|\sum^{p-1}_{k=0}\left(
                               \begin{array}{cc}
                                 1 & k \\
                                 0 & p \\
                               \end{array}
                             \right)\\
  &=g|\sum^{p-1}_{k=0}\left(
                               \begin{array}{cc}
                                 1 & k \\
                                 0 & p \\
                               \end{array}
                             \right)-p^{-1}\cdot\chi(p)^{-1}\cdot g|\sum^{p-1}_{k=0}\left(
                               \begin{array}{cc}
                                 p & pk \\
                                 0 & p \\
                               \end{array}
                             \right)\\
  &=\CT^{(N)}_p(g)-f|\gamma_p-\chi(p)^{-1}\cdot g,
\end{align*}
which proves (2). The proof of (3) is similar so we omit it here.
\end{proof}

\subsection{}From now on we fix an odd positive square-free integer $D$ and a positive divisor $C$ of $D$. The assumption that $D$ is odd will not be needed until the final section. Define $\CH(DC)$ to be the set of all triples $(M,L,\chi)$ satisfying:
\begin{itemize}
  \item $1\leq M,L\mid D$ with $M\neq1$;
  \item $D\mid ML\mid DC$;
  \item $\chi$ is a Dirichlet character modulo $(M,L)$.
\end{itemize}

\begin{lem}\label{1.2}
$\#\CH(DC)=\dim_{\BC}\ \CE_2(\Gamma_0(DC),\BC)$.
\end{lem}
\begin{proof}
It is well known that the number of cusps of $X_0(DC)$ is equal to $\sum_{1\leq d\mid DC}\varphi(d,{DC}/{d})$ (see \S2.1 of \cite{L}), and hence $\dim_{\BC}\CE_2(\Gamma_0(DC),\BC)=\sum_{1< d\mid DC}\varphi(d,{DC}/{d})$. Here $\varphi(d,{DC}/{d})$ means applying Euler's $\varphi$-function to the greatest common divisor of $d$ and $DC/d$. Thus we only need to prove that $\#\CH(DC)=\sum_{1< d\mid DC}\varphi(d,{DC}/{d})$. We will first prove this when $C=D$. For any positive divisor $d$ of $D^2$, we can associate the following two positive integers
\[M\colonequals\sqrt{d\cdot(d,\frac{D^2}{d})},\ L\colonequals\sqrt{\frac{D^2}{d}\cdot(d,\frac{D^2}{d})}\]
such that $1\leq M,L\mid D$ and $D\mid ML\mid D^2$. Conversely, to any pair of integers $M$ and $L$ with $1\leq M,L\mid D$ and $D\mid ML\mid D^2$, we can associate a positive divisor $d$ of $D$ as
\[d\colonequals\left[\frac{M}{(M,L)}\right]^2\cdot(M,L).\]
It is easy to see that the above establishes a bijection between $\{d:1\leq d\mid D^2\}$ and the set of all pair of integers $M$ and $L$ with $1\leq M,L\mid D$ and $D\mid ML\mid D^2$. Moreover, under this bijection, the divisor $1$ of $D^2$ corresponds to the pair $M=1$ and $L=D$, and we have $(d,{D^2}/{d})=(M,L)$ if $d$ corresponds to $M$ and $L$. It follows that there is a bijection between $\{(d,\chi)|1<d\mid D^2,\chi:({\BZ}/{(d,{D^2}/{d})\cdot\BZ})^\times\rightarrow\BC^\times\}$ and $\CH(D^2)$ which proves the lemma in this situation.

In general, since $DC=\frac{D}{C}\cdot C^2$, any positive divisor $d$ of $DC$ can be uniquely decomposed as $d=d_0\cdot d'$ with $1\leq d_0\mid\frac{D}{C}$ and $1\leq d'\mid C^2$. If such a positive divisor $d'$ of $C^2$ corresponds to a pair of integer $m$ and $\ell$ with $1\leq m,\ell\mid C$ and $C\mid m\ell\mid C^2$ as above, then we can associate with $d$ the pair of integers $M=d_0\cdot m$ and $\frac{DC}{d_0}\cdot\ell$ which satisfies $1\leq M,L\mid D$ and $D\mid ML\mid DC$. This establishes a bijection between $\{d|1\leq d\mid DC\}$ and the set of all pair of integers $M$ and $L$ with $1\leq M,L\mid D$ and $D\mid ML\mid DC$. Moreover, we have $1\mid D^2$ corresponds to the pair $M=1$ and $L=D$, and $(d,\frac{DC}{d})=(M,L)$ if $d$ corresponds to $M$ and $L$. It follows that there is a bijection between $\{(d,\chi)|1<d\mid D^2,\chi:\left(\BZ/(d,{DC}/{d})\cdot\BZ\right)^\times\rightarrow\BC^\times\}$ and $\CH(DC)$ which completes the proof the lemma.
\end{proof}

\begin{defn}\label{key definition}
For any Dirichlet character $\chi$ of conductor $\ff_\chi$, let
\begin{align*}
  E_\chi\colonequals-\frac{1}{2g(\chi)}\sum_{a\in(\BZ/\ff_\chi\BZ)^\times}\sum_{b\in(\BZ/\ff^2_\chi\BZ)^\times}\chi(a)\cdot\chi(b)\cdot\phi_{(\frac{a}{\ff_\chi},\frac{b}{\ff^2_\chi})},
\end{align*}
where $g(\chi)$ is the Gauss sum of $\chi$. Then we define
\[E_{M,L,\chi}\colonequals[\frac{L}{\ff_\chi}]^-_\chi\circ[\frac{M}{\ff_\chi}]^+_\chi(E_\chi),\]
for any $(M,L,\chi)\in\CH(DC)$.
\end{defn}

\subsection{}Let $\delta_\chi=1$ or $0$ according to $\chi$ is trivial or not. It follows from (2.1) that
\begin{align*}
E_\chi=-\frac{\delta_\chi}{4\pi i(z-\overline{z})}-\frac{1}{g(\chi)}\sum_{a\in({\BZ}/{\ff_\chi\BZ})^\times}\sum_{b\in({\BZ}/{\ff_\chi^2\BZ})^\times}\chi(a)\cdot\chi(b)\cdot \left(\frac{1}{4}B_2(\frac{a}{\ff_\chi})-P_{(\frac{a}{\ff_\chi},\frac{b}{\ff_\chi^2})}\right).
\end{align*}
 Moreover, by (2.2) we have that
\begin{align*}
\sum_{a\in({\BZ}/{\ff_\chi\BZ})^\times}\sum_{b\in({\BZ}/{\ff_\chi^2\BZ})^\times}\chi(a)\cdot\chi(b)\cdot P_{(\frac{a}{\ff_\chi},\frac{b}{\ff_\chi^2})}&=\sum^{\infty}_{k,m=1}\frac{k\chi(k)}{\ff_\chi}\left(\sum_{y\in({\BZ}/{\ff_\chi^2\BZ})^{\times}}\chi(y)e^{2\pi i\frac{my}{\ff_\chi^2}}\right)e^{2\pi i\frac{mk}{\ff_\chi}z}\\
&=\sum^{\infty}_{k,m=1}\frac{k\chi(k)}{\ff_\chi}\left(\sum_{y\in({\BZ}/{\ff_\chi^2\BZ})^{\times}}\chi(y)e^{2\pi i\frac{my}{\ff_\chi}}\right)e^{2\pi imkz}\\
&=g(\chi)\sum^{\infty}_{k,m=1}k\cdot\chi(k)\cdot\chi^{-1}(m)\cdot e^{2\pi imkz},
\end{align*}
where a Dirichlet character is extended to a function on $\BZ$ in the usual way, and therefore
\begin{align}
  E_\chi=-\frac{\delta_\chi}{4\pi i(z-\overline{z})}+a_0(E_\chi;[\infty])+\sum^{\infty}_{n=1}\sigma_{\chi}(n)\cdot \fq^n,
\end{align}
with
\begin{align}
a_0(E_\chi;[\infty])=
\begin{cases}
-\frac{1}{24} &\text{if}\ \chi=1\\
\ \ 0 &\text{otherwise}\
\end{cases}
\end{align}
and
\begin{align}
\sigma_\chi(n):=
\sum_{1\leq d\mid n}d\cdot\chi(d)\cdot\chi^{-1}({n}/{d}).
\end{align}

\begin{lem}\label{normalized}
For any $(M,L,\chi)\in\CH(DC)$, $E_{M,L,\chi}$ belongs to $\CE_2(\Gamma_0(DC),\BC)$ and is normalized.
\end{lem}
\begin{proof}
If $\chi$ is nontrivial so that $\ff_\chi>1$, then $E_\chi$ belongs to $\CE_2(\Gamma_0(\ff^2_\chi),\BC)$ by (2.3). Since $ML$ divides $DC$, it follows form Lemma~\ref{1.0} that
\begin{align*}
  E_{M,L,\chi}&\in\CE_2(\Gamma_0(ML),\BC)\\
  &\subseteq\CE_2(\Gamma_0(DC),\BC).
\end{align*}
If $(M,L,1)\in\CH(DC)$, we take a prime divisor $p$ of $M$, which is possible because $M>1$. Since
\begin{align*}
  [p]^+(E_1)&=-\frac{1}{2}\left(\phi_{(0,0)}-\phi_{(0,0)}|\gamma_p\right)\\
  &=-\frac{1}{2}\phi_{(0,0)}+\frac{1}{2}\sum_{b\in\BZ/p\BZ}\phi_{(0,\frac{b}{p})}\\
  &=\frac{1}{2}\sum_{b\in(\BZ/p\BZ)^\times}\phi_{(0,\frac{b}{p})},
\end{align*}
it follows that $[p]^+(E_1)$ belongs to $\CE_2(\Gamma_0(p),\BC)$. And we find similarly as above that $E_{M,L,1}$ belongs to $\CE_2(\Gamma_0(DC),\BC)$. Finally, since both $[M/\ff_\chi]^+_\chi$ and $[L/\ff_\chi]^-_\chi$ preserve the first terms of Fourier expansions, we find that $a_1(E_{M,L,\chi};[\infty])=a_1(E_\chi;[\infty])=1$ and hence complete the proof.
\end{proof}

\begin{lem}\label{lem2}
For any Dirichlet character $\chi$ of conductor $\ff_\chi$, we have that
\[\CT^{(\ff^2_\chi)}_{\ell}(E_{\chi})=
\begin{cases}
\left(\chi(\ell)^{-1}+\ell\cdot\chi(\ell)\right)\cdot E_\chi &\text{if}\ \ell\nmid \ff_\chi\\
0 &\text{if}\ \ell\mid \ff_\chi\\
\end{cases}
\]
\end{lem}
\begin{proof}
Let $\ell$ be a prime such that $(\ell,\ff_\chi)=1$. By Proposition 2.4.7 of \cite{St}, for any integers $x,y$ prime to $\ff_\chi$, we have
\[\CT^{(\ff_\chi^2)}_\ell(\phi_{(\frac{x}{\ff_\chi},\frac{y}{\ff_\chi^2})})=\phi_{(\frac{x}{\ff_\chi},\frac{\ell y}{\ff_\chi^2})}+\ell\cdot\phi_{(\frac{\ell'x}{\ff_\chi},\frac{ y}{\ff_\chi^2})},\]
where $\ell'$ is an integer such that $\ell\ell'\equiv1\pmod{\ff_\chi}$. It follows that
\[\CT^{(\ff_\chi^2)}_{\ell}(E_{\chi})=\left(\chi(\ell)^{-1}+\ell\cdot\chi(\ell)\right)\cdot E_\chi.\]
On the other hand, by the distribution law (2.4), we have
\[E_\chi=-\frac{1}{2g(\chi)}\sum_{x,y\in({\BZ}/{\ff_\chi\BZ})^\times}\chi(x)\cdot\chi(y)\cdot\phi_{(\frac{x}{\ff_\chi},\frac{y}{\ff_\chi})}|\left(
                                                                                                                             \begin{array}{cc}
                                                                                                                               \ff_\chi & 0 \\
                                                                                                                               0 & 1 \\
                                                                                                                             \end{array}
                                                                                                                           \right),
\]
so, if $\ell$ is a prime divisor of $\ff_\chi$, then
\begin{align*}
  \CT^{(\ff_\chi^2)}_\ell(E_\chi)&=-\frac{1}{2g(\chi)}\sum_{x,y\in({\BZ}/{\ff_\chi\BZ})^\times}\chi(x)\cdot\chi(y)\cdot\phi_{(\frac{x}{\ff_\chi},\frac{y}{\ff_\chi})}|\left(
                                                                                                                             \begin{array}{cc}
                                                                                                                               \ff_\chi & 0 \\
                                                                                                                               0 & 1 \\
                                                                                                                             \end{array}
                                                                                                                           \right)\sum^{\ell-1}_{k=0}\left(
                                                                                                                                                       \begin{array}{cc}
                                                                                                                                                         1 & k \\
                                                                                                                                                          0& \ell \\
                                                                                                                                                       \end{array}
                                                                                                                                                     \right)\\
  &=-\frac{1}{2g(\chi)}\sum_{x,y\in({\BZ}/{\ff_\chi\BZ})^\times}\chi(x)\cdot\chi(y)\cdot\phi_{(\frac{x}{\ff_\chi},\frac{y}{\ff_\chi})}|\sum^{\ell-1}_{k=0}\left(
                                                                                                                                                       \begin{array}{cc}
                                                                                                                                                         1 & \frac{\ff_\chi}{\ell}k \\
                                                                                                                                                          0& 1 \\
                                                                                                                                                       \end{array}
                                                                                                                                                     \right)\left(
                                                                                                                             \begin{array}{cc}
                                                                                                                               \ff_\chi & 0 \\
                                                                                                                               0 & \ell \\
                                                                                                                             \end{array}
                                                                                                                           \right)\\
  &=-\frac{1}{2g(\chi)}\sum_{x,y\in({\BZ}/{\ff_\chi\BZ})^\times}\chi(x)\cdot\chi(y)\sum^{\ell}_{k=0}\phi_{(\frac{x}{\ff_\chi},\frac{y}{\ff_\chi}+\frac{xk}{\ell})}|\left(
                                                                                                                                                     \begin{array}{cc}
                                                                                                                                                       \ff_\chi & 0 \\
                                                                                                                                                       0 & \ell \\
                                                                                                                                                     \end{array}
                                                                                                                                                   \right)=0.
\end{align*}
\end{proof}

\begin{prop}\label{eigen}
For any $(M,L,\chi)\in\CH(DC)$, we have
\begin{align*}
  \CT^{(DC)}_{\ell}(E_{M,L,\chi})=
\begin{cases}
\left(\chi(\ell)^{-1}+\ell\cdot\chi(\ell)\right)\cdot E_{M,L,\chi} &\text{if}\ \ell\nmid D\\
\chi(\ell)^{-1}\cdot E_{M,L,\chi} &\text{if}\ \ell\mid \frac{M}{(M,L)}\\
\ell\cdot\chi(\ell)\cdot E_{M,L,\chi} &\text{if}\ \ell\mid\frac{L}{(M,L)}\\
0 &\text{if}\ \ell\mid(M,L).
\end{cases}
\end{align*}
In particular $\CE_2((DC),\BC)=\bigoplus_{(M,L,\chi)\in\CH(DC)}\BC\cdot E_{M,L,\chi}$.
\end{prop}
\begin{proof}
It suffices to prove the first assertion. By Lemma~\ref{1.1}(1) and Lemma~\ref{lem2}, if $\ell$ is a prime not dividing $D$, then
\begin{align*}
  \CT^{(DC)}_\ell(E_{M,L,\chi})&=[\frac{L}{\ff_\chi}]^-_\chi\circ[\frac{M}{\ff_\chi}]^+_\chi\circ\CT^{(\ff_\chi^2)}_\ell(E_\chi)\\
  &=\left(\chi(\ell)^{-1}+\ell\cdot\chi(\ell)\right)\cdot E_{M,L,\chi}.
\end{align*}
By Lemma~\ref{1.1}(2), if $\ell$ is a prime divisor of ${M}/{(M,L)}$, then
\begin{align*}
 \CT^{(DC)}_\ell(E_{M,L,\chi})&=[\frac{L}{\ff_\chi}]^-_\chi\circ[\frac{M}{\ff_\chi\ell}]^+_\chi\circ\CT^{(\ff_\chi^2\ell)}_\ell\circ[\ell]^+_\chi(E_\chi)\\
 &=[\frac{L}{\ff_\chi}]^-_\chi\circ[\frac{M}{\ff_\chi\ell}]^+_\chi\circ(\chi(\ell)^{-1}-\gamma_\ell)(E_\chi)\\
 &=\chi(\ell)^{-1}\cdot E_{M,L,\chi}.
\end{align*}
The proofs for those prime divisors of  $\frac{L}{(M,L)}$ and $\frac{(M,L)}{\ff_\chi}$ are similar to the above, so we omit it here. Finally, if $\ell$ is a prime divisor of $\ff_\chi$, then we find by Lemma~\ref{1.1}(1) and Lemma~\ref{lem2} that
\[\CT^{(DC)}_\ell(E_{M,L,\chi})=[\frac{L}{\ff_\chi}]^-_\chi\circ[\frac{M}{\ff_\chi}]^+_\chi\circ\CT^{(\ff_\chi^2)}_\ell(E_\chi)=0,\]
which completes the proof of the proposition.
\end{proof}

\section{Orders of quadratic cuspidal subgroups}
\subsection{}For any triple $(M,L,\chi)$ in $\CH(DC)$, denote by $\CC^{(DC)}_{M,L,\chi}$ the cuspidal subgroup $C_{\Gamma_0(DC)}(E_{M,L,\chi})$ associated to the Eisenstein series $E_{M,L,\chi}$ as in Definition~\ref{def}. When $\chi$ is a quadratic character, we call $\CC^{(DC)}_{M,L,\chi}$ a \emph{quadratic} cuspidal subgroup.

Recall that, for a general positive integer $N$, we have representatives for cusps on $X_0(N)$ of the form $[\frac{dx}{N}]$, where $d\mid N$, $d>0$ and $(x,d)=1$ with $x$ taken modulo $(d,N/d)$. And such a cusp $[\frac{dx}{N}]$ is defined over $\BQ(\mu_c)$ with $c=(d,N/d)$ (See \S2.1 of \cite{L}). Put $N=DC$. Then any positive divisor of $DC$ is of the form $rs^2t$, where $r\mid\frac{D}{C}$, $s,t\mid C$ and $(s,t)=1$. So we obtain a full set of representatives $\{[\frac{rs^2tx}{DC}]\}$ for cusps on $X_0(DC)$, where $r,s,t$ as above and $(x,rst)=1$ with $x$ taken modulo $t$.

\begin{lem}\label{change}
Let $p$ be a prime divisor of $D$ and $[\frac{rs^2tx}{DC}]$ a cusp of $X_0(DC)$, then:
\begin{enumerate}
  \item If $p\mid r$, then $[\frac{rs^2tx}{DC}]=[\frac{(r/p)s^2tx}{DC/p}]$ in $X_0(DC/p)$;
  \item If $p\mid s$, then $[\frac{rs^2tx}{DC}]=[\frac{r(s/p)^2tx}{DC/p^2}]$ in $X_0(DC/p^2)$;
  \item If $p\mid t$, then $[\frac{rs^2tx}{DC}]=[\frac{r(s/p)^2(t/p)\cdot(px)}{DC/p^2}]$ in $X_0(DC/p^2)$;
  \item If $p\mid \frac{D}{Cr}$, then $[\frac{rs^2tx}{DC}]=[\frac{rs^2t\cdot(px)}{DC/p}]$ in $X_0(DC/p)$;
  \item If $p\mid \frac{C}{st}$, then $[\frac{rs^2tx}{DC}]=[\frac{rs^2t\cdot(p^2x)}{DC/p^2}]$ in $X_0(DC/p^2)$.
\end{enumerate}
\end{lem}
\begin{proof}
The first two assertions are obvious. Since the proofs of last three assertions are similar, we will only give that of (3). If $[\frac{rs^2tx}{DC}]=[\frac{r's'^2t'x'}{DC/p^2}]$ in $X_0(DC/p^2)$, then there exists some $\gamma=\left(
                                                                                                         \begin{array}{cc}
                                                                                                           \alpha & \beta \\
                                                                                                           \frac{DC}{p^2}\delta & \omega \\
                                                                                                         \end{array}
                                                                                                       \right)\in \Gamma_0(\frac{DC}{p^2})
$ sending the former point to the latter one, and hence
\[r's'^2t'x'=rs^2(t/p)\cdot\frac{x\alpha+\beta\frac{DC}{rs^2t}}{\delta rs^2(t/p)x+\omega p}.\]
Since $\delta rs^2(t/p)x+\omega p$ is prime to $rs^2(t/p)$, we find that $r,s,t/p$ divides $r',s',t'$ respectively. And it follows
\[\frac{r'}{r}\cdot\frac{s'^2}{s^2}\cdot\frac{t'}{t/p}\cdot x'=\frac{x\alpha+\beta\frac{DC}{rs^2t}}{\delta rs^2(t/p)x+\omega p}.\]
If $q$ is a prime divisor of $r's't'$ (so that $q\neq p$ as $p\nmid t'$) with $(q,rst)=1$, then $x\alpha+\beta\frac{DC}{rs^2t}$ would be a $q$-adic unit, which contradicts to the above equation and hence proves the assertion.
\end{proof}

Let $K$ be a positive divisor of $D$ and $\alpha$ a positive divisor of $K$. We find by Lemma~\ref{change} that: if $(K,rst)=1$, then we have
\begin{align}\label{4.1}
[\frac{rs^2t\alpha x}{DC}]=[\frac{rs^2t(\frac{K(K,C)}{\alpha}x)}{DC/K(K,C)}]\in X_0(\frac{DC}{K(K,C)});
\end{align}
and if $K\mid t$, then we have
\begin{align}\label{4.2}
[\frac{rs^2t\alpha x}{DC}]=[\frac{rs^2(\frac{t}{K})(\frac{K}{\alpha}x)}{DC/K^2}]\in X_0(\frac{DC}{K^2}).
\end{align}
We leave the verifications of the above equalities to the reader.

\subsection{}For any cusp $[\frac{a}{c}]$ of $X_0(N)$ with $(a,b)=1$, choose some $\gamma=\left(
                                             \begin{array}{cc}
                                               a & b \\
                                               c & d \\
                                             \end{array}
                                           \right)\in\SL_2(\BZ)$ such that $\gamma([\infty])=[\frac{a}{c}]$. Moreover, for any prime $p$, we can and will always assume $p\mid d$ when $p\nmid c$, so that
\[\gamma_p\cdot\gamma=
\begin{cases}
\left(
           \begin{array}{cc}
             a & pb \\
             c/p & d \\
           \end{array}
         \right)\left(
                  \begin{array}{cc}
                    p & 0 \\
                    0 & 1 \\
                  \end{array}
                \right)
 &\text{if}\ p\mid c\\
\left(
           \begin{array}{cc}
             ap & b \\
             c & d/p \\
           \end{array}
         \right)\left(
                  \begin{array}{cc}
                    1 & 0 \\
                    0 & p \\
                  \end{array}
                \right)
&\text{if}\ p\nmid c.
\end{cases}
\]
Take $g\in\CM_2(\Gamma_0(N),\BC)$. If $\chi$ is of conductor $\ff_\chi$ and $p$ is a prime not dividing $\ff_\chi$, then
\[a_0([p]^+_\chi(g);[\frac{a}{c}])=
\begin{cases}
a_0(g;[\frac{a}{c}])-p\cdot\chi(p)\cdot a_0(g;[\frac{ap}{c}])
 &\text{if}\ p\mid c\\
a_0(g;[\frac{a}{c}])-p^{-1}\cdot\chi(p)\cdot a_0(g;[\frac{ap}{c}])
&\text{if}\ p\nmid c,
\end{cases}
\]
and
\[a_0([p]^-_\chi(g);[\frac{a}{c}])=
\begin{cases}
a_0(g;[\frac{a}{c}])-\chi(p)^{-1}\cdot a_0(g;[\frac{ap}{c}])
 &\text{if}\ p\mid c\\
a_0(g;[\frac{a}{c}])-p^{-2}\cdot\chi(p)^{-1}\cdot a_0(g;[\frac{ap}{c}])
&\text{if}\ p\nmid c.
\end{cases}
\]
Thus, for any positive square-free integer $K$ prime to $\ff_\chi$, we find by induction that
\begin{align}
  a_0([K]^+_\chi(g);[\frac{a}{c}])=
\begin{cases}
\sum_{1\leq\alpha\mid K}(-1)^{\nu(\alpha)}\cdot\alpha\cdot\chi(\alpha)\cdot a_0(g;[\frac{\alpha a}{c}])
 &\text{if}\ K\mid c\\
\sum_{1\leq\alpha\mid K}(-1)^{\nu(\alpha)}\cdot\alpha^{-1}\cdot\chi(\alpha)\cdot a_0(g;[\frac{\alpha a}{c}])
&\text{if}\ (K,c)=1,
\end{cases}
\end{align}
and
\begin{align}
  a_0([K]^-_\chi(g);[\frac{a}{c}])=
\begin{cases}
\sum_{1\leq\alpha\mid K}(-1)^{\nu(\alpha)}\cdot\chi(\alpha)^{-1}\cdot a_0(g;[\frac{\alpha a}{c}])
 &\text{if}\ K\mid c\\
\sum_{1\leq\alpha\mid K}(-1)^{\nu(\alpha)}\cdot\alpha^{-2}\cdot\chi(\alpha)^{-1}\cdot a_0(g;[\frac{\alpha a}{c}])
&\text{if}\ (K,c)=1.
\end{cases}
\end{align}

Now we come back to our situation. For any cusp $[\frac{s^2tx}{\ff_\chi^2}]\in X_0(\ff_\chi^2)$, choose some $\gamma=\left(
                                                                                                                     \begin{array}{cc}
                                                                                                                       x & u \\
                                                                                                                       {\ff_\chi^2}/{s^2t} & v \\
                                                                                                                     \end{array}
                                                                                                                   \right)
$ in $\SL_2(\BZ)$ such that $\gamma[\infty]=[\frac{s^2tx}{\ff_\chi^2}]$. Then it follows from (2.1) and (2.4) that
\begin{align*}
  a_0(E_\chi;[\frac{s^2tx}{\ff_\chi^2}])&=a_0(E_\chi|\gamma;[\infty])\\
  &=-\frac{1}{4g(\chi)}\sum_{a\in({\BZ}/{\ff_\chi\BZ})^\times}\sum_{b\in({\BZ}/{\ff_\chi^2\BZ})^\times}\chi(a)\cdot\chi(b)\cdot B_2(\frac{xa}{\ff_\chi}+\frac{b}{s^2t})\\
  &=-\frac{1}{4g(\chi)}\sum_{b\in({\BZ}/{\ff_\chi^2\BZ})^\times}\chi(b)\left(\sum_{a\in({\BZ}/{\ff_\chi\BZ})^\times}\chi(a)\cdot B_2(\frac{xa}{\ff_\chi}+\frac{b}{s^2t})\right).
\end{align*}
It is clear that the function in the above bracket depends only on $b$ modulo $s^2t$. Since $\chi$ is primitive of conductor $\ff_\chi$, we find that $a_0(E_\chi;[\frac{s^2tx}{\ff_\chi^2}])=0$ unless $st=\ff_\chi$. Moreover, if $st=\ff_\chi$, then
\begin{align*}
  a_0(E_\chi;[\frac{s^2tx}{\ff_\chi^2}])&=-\frac{1}{4g(\chi)}\sum_{a\in({\BZ}/{\ff_\chi\BZ})^\times}\sum_{b\in({\BZ}/{\ff_\chi^2\BZ})^\times}\chi(a)\cdot\chi(b)\cdot B_2(\frac{xa}{\ff_\chi}+\frac{b}{s\ff_\chi})\\
  &=-\frac{\chi(x)^{-1}}{4g(\chi)}\sum_{a\in({\BZ}/{\ff_\chi\BZ})^\times}\chi(a)\left(\sum_{b_0,k\in({\BZ}/{\ff_\chi\BZ})^\times}\chi(b_0)\cdot B_2(\frac{as+b_0+k\ff_\chi}{s \ff_\chi})\right),
\end{align*}
with the function in the bracket depends only on $a$ modulo ${{\ff_\chi}/{s}}$ and hence equals zero unless $s=1$. It follows that
\begin{align}
a_0(E_\chi;[\frac{s^2tx}{\ff_\chi^2}])=
\begin{cases}
\chi(x)^{-1}\cdot n_\chi&\text{if}\ s=1~\text{and}~t=\ff_\chi\\
0 &\text{otherwise,}\
\end{cases}
\end{align}
where
\[n_\chi:=-\frac{\ff_\chi}{4g(\chi)}\sum_{a,b\in{\BZ}/{\ff_\chi\BZ}}\chi(a)\cdot\chi(b)\cdot B_2(\frac{a+b}{\ff_\chi}).\]
In particular, for any integer $\alpha$ prime to $\ff_\chi$, we have
\begin{align}
  a_0(E_\chi;[\frac{s^2t(\alpha x)}{\ff_\chi^2}])=\chi(\alpha)^{-1}\cdot a_0(E_\chi;[\frac{s^2tx}{\ff_\chi^2}]).
\end{align}

\begin{lem}\label{constant1}
If $\chi$ is a quadratic character modulo $C$, then
\[
a_0(E_{D,\ff_\chi,\chi};[\frac{rs^2tx}{DC}])=
\begin{cases}
 n_\chi\cdot\varphi(\frac{D}{\ff_\chi})\cdot(-1)^{\nu(\frac{D}{{\ff_\chi}rs})}\cdot\chi(\frac{DC}{{\ff_\chi}rs^2tx})\cdot(rs)^{-1}&\text{if}\ (s,\ff_\chi)=1~\text{and}~\ff_\chi\mid t\\
0 &\text{otherwise}\
\end{cases}
\]
for any cusp $[\frac{rs^2tx}{DC}]$ in $X_0(DC)$, where $\ff_\chi$ is the conductor of $\chi$. In particular, for any integer $\alpha$ prime to $D$, we have
\[a_0(E_{D,\ff_\chi,\chi};[\frac{rs^2t(\alpha x)}{DC}])=\chi(\alpha)\cdot a_0(E_{D,\ff_\chi,\chi};[\frac{rs^2tx}{DC}])).\]
\end{lem}
\begin{proof}
For any cusp $[\frac{rs^2tx}{DC}]$, let
\begin{align*}
  \begin{cases}
    K_r\colonequals({D}/{\ff_\chi},r)=r\\
    K_s\colonequals({D}/{\ff_\chi},s)\\
    K_t\colonequals({D}/{\ff_\chi},t),
  \end{cases}
\end{align*}
so that we have ${D}/{\ff_\chi}=K_r K_s K_t K$. Then it follows from (4.1), (4.2) and (4.3) that
\begin{align*}
  a_0(E_{D,\ff_\chi,\chi};[\frac{rs^2tx}{DC}])&=\sum_{1\leq\alpha\mid K}(-1)^{\nu(\alpha)}\cdot\chi(\alpha)\cdot\alpha\cdot a_0(E_{\frac{D}{K},\ff_\chi,\chi};[\frac{rs^2t(\frac{K(K,C)}{\alpha}x)}{DC/K(K,C)}])\\
  &=\sum_{1\leq\alpha\mid K,1\leq\alpha_t\mid K_t}(-1)^{\nu(\alpha\alpha_t)}\cdot\chi(\alpha\alpha_t)\cdot\alpha\alpha_t\cdot  a_0(E_{\frac{D}{K_tK},\ff_\chi,\chi};[\frac{rs^2(\frac{t}{K_t})(\frac{K_tK(K,C)}{\alpha_t\alpha}x)}{DC/K^2_tK(K,C)}]),
\end{align*}
and we find by (4.3) and Lemma~\ref{change}(1),(2) that
\begin{align*}
  a_0(E_{D,\ff_\chi,\chi};[\frac{rs^2tx}{DC}])&=\sum(-1)^{\nu(\alpha_r\alpha_s\alpha_t\alpha)}\cdot\chi(\alpha_r\alpha_s\alpha_t\alpha)\cdot\frac{\alpha_t\alpha}{\alpha_r\alpha_s}\cdot  a_0(E_\chi;[\frac{(\frac{s}{K_s})^2(\frac{t}{K_t})(\frac{K_tK(K,C)}{\alpha_t\alpha}\alpha_r\alpha_s x)}{\ff_\chi^2}])\\
  &=\chi(K_tK(K,C))\cdot\sum(-1)^{\nu(\alpha_r\alpha_s\alpha_t\alpha)}\cdot\frac{\alpha_t\alpha}{\alpha_r\alpha_s}\cdot  a_0(E_\chi;[\frac{(\frac{s}{K_s})^2(\frac{t}{K_t})x}{\ff_\chi^2}]),
\end{align*}
where $\alpha_r,\alpha_s,\alpha_t$ and $\alpha$ run through all positive divisors of $K_r,K_s,K_t$ and $K$ respectively. Thus, by (4.5) and (4.6), we have
\[a_0(E_{D,\ff_\chi,\chi};[\frac{rs^2tx}{DC}])=\chi(K_tK(K,C))\cdot\prod_{p\mid K_rK_s}(1-p^{-1})\cdot\prod_{p\mid K_tK}(1-p)\cdot a_0(E_\chi;[\frac{(\frac{s}{K_s})^2(\frac{t}{K_t})x}{\ff_\chi^2}]),\]
which is zero unless $s=K_s$ and $\ff_\chi K_t\mid t$, or equivalently, $(s,\ff_\chi)=1$ and $\ff_\chi\mid t$. This proves the lemma because $K_rK_s=rs$, $K_tK=\frac{D}{rs\ff_\chi}$ and $(K,C)=\frac{C}{st}$ when these conditions are satisfied.
\end{proof}

\begin{lem}\label{constant2}
Notations as in the above lemma, then we have
\[a_0(E_{M,\frac{D\ff_\chi}{M},\chi};[\frac{rs^2tx}{DC}])=
\begin{cases}
n_\chi\varphi(\frac{D}{\ff_\chi})\psi(\frac{D}{M})\frac{M}{rsD} c_\chi[\frac{rs^2tx}{DC}] &\text{if}\ (s,\ff_\chi)=1, \frac{D}{M}\mid rs\ and\ \ff_\chi\mid t\\
0
&\text{otherwise}\
\end{cases}\]
for any cusp $[\frac{rs^2tx}{DC}]$ in $X_0(DC)$, where
\[c_\chi[\frac{rs^2tx}{DC}]\colonequals{(-1)^{\nu(\frac{D}{\ff_\chi rs})}\chi(\frac{DC}{\ff_\chi rs^2tx})}.\]
In particular, for any integer $\alpha$ prime to $D$, we have
\[a_0(E_{M,\frac{D\ff_\chi}{M},\chi};[\frac{rs^2t(\alpha x)}{DC}])=\chi(\alpha)\cdot a_0(E_{M,\frac{D\ff_\chi}{M},\chi};[\frac{rs^2tx}{DC}])).\]
\end{lem}
\begin{proof}
For any cusp $[\frac{rs^2tx}{DC}]$, let
\begin{align*}
  \begin{cases}
    H_r\colonequals({D}/{M},r)\\
    H_s\colonequals({D}/{M},s)\\
    H_t\colonequals({D}/{M},t),
  \end{cases}
\end{align*}
so that we have ${D}/{M}=H_r H_s H_t H$. Then it follows from (4.1), (4.2) and (4.4) that
\begin{align*}
  a_0(E_{M,\ff_\chi\cdot\frac{D}{M},\chi};[\frac{rs^2tx}{DC}])&=\sum_{1\leq\alpha\mid H}(-1)^{\nu(\alpha)}\cdot\chi(\alpha)\cdot a_0(E_{M,\ff_\chi\cdot\frac{D}{MH},\chi};[\frac{rs^2t(\frac{H(H,C)}{\alpha}x)}{DC/H(H,C)}])\\
   &=\sum_{1\leq\alpha\mid H,1\leq\alpha_t\mid H_t}(-1)^{\nu(\alpha_t\alpha)}\cdot\chi(\alpha_t\alpha)\cdot
   a_0(E_{M,\ff_\chi\cdot\frac{D}{MH_tH},\chi};[\frac{rs^2(\frac{t}{H_t})(\frac{H_tH(H,C)}{\alpha_t\alpha}x)}{DC/H^2_tH(H,C)}]).
\end{align*}
So we find by (4.4), Lemma~\ref{change}(1),(2) and Lemma~\ref{constant1} that
\begin{align*}
  &a_0(E_{M,\ff_\chi\cdot\frac{D}{M},\chi};[\frac{rs^2tx}{DC}])\\
  &=\sum(-1)^{\nu(\alpha_r\alpha_s\alpha_t\alpha)}\cdot\chi(\alpha_r\alpha_s\alpha_t\alpha)\cdot(\alpha_r\alpha_s)^{-2}\cdot  a_0(E_{M,\ff_\chi,\chi};[\frac{(\frac{r}{H_r})(\frac{s}{H_s})^2(\frac{t}{H_t})(\frac{H_tH(H,C)}{\alpha_t\alpha}\alpha_r\alpha_s x)}{M\cdot(M,C)}])\\
  &=\chi(HH_t(H,C))\sum(-1)^{\nu(\alpha_r\alpha_s\alpha_t\alpha)}\cdot(\alpha_r\alpha_s)^{-2}\cdot  a_0(E_{M,\ff_\chi,\chi};[\frac{(\frac{r}{H_r})(\frac{s}{H_s})^2(\frac{t}{H_t})x}{M\cdot(M,C)}]),
\end{align*}
where $\alpha_r,\alpha_s,\alpha_t$ and $\alpha$ run through all positive divisors of $H_r,H_s,H_t$ and $H$ respectively. The above sum is zero unless $H_t=H=1,(s,\ff_\chi)=1$ and $\ff_\chi\mid t$, or equivalently, $\frac{D}{M}\mid rs,(s,\ff_\chi)=1$ and $\ff_\chi\mid t$. When these conditions are satisfied, then the assertion follows from the previous Lemma.
\end{proof}

\begin{prop}\label{constant3}
For any $(M,L,\chi)$ in $\CH(DC)$ with $\chi$ a quadratic character, we have
\begin{align*}
a_0(E_{M,L,\chi};[\frac{rs^2tx}{DC}])=
\begin{cases}
n_\chi{\varphi(\frac{D}{\ff_\chi})\psi(\frac{L}{\ff_\chi})}\frac{\ff_\chi}{rsL}\frac{\varphi(s,M,L)}{(s,M,L)}c_\chi[\frac{rs^2tx}{DC}]&\text{if}\ (s,\ff_\chi)=1, (M,L)\mid st\ and\ \frac{D}{M}\mid rs\\
0 &\text{otherwise}
\end{cases}
\end{align*}
for any cusp $[\frac{rs^2tx}{DC}]$ in $X_0(DC)$, where $\ff_\chi$ is the conductor of $\chi$.
\end{prop}
\begin{proof}
It remains to consider the case when $(M,L)\neq \ff_\chi$. For any cusp $[\frac{rs^2tx}{DC}]$, let
\begin{align*}
  \begin{cases}
    W_s\colonequals(\frac{(M,L)}{\ff_\chi},s)\\
    W_t\colonequals(\frac{(M,L)}{\ff_\chi},t),
  \end{cases}
\end{align*}
so that ${(M,L)}/{\ff_\chi}=W_s W_t W$. Then it follows from (4.4) that
\begin{align*}
  a_0(E_{M,L,\chi};[\frac{rs^2tx}{DC}])&=\sum(-1)^{\nu(\alpha)}\cdot\chi(\alpha)\cdot a_0(E_{M,\ff_\chi\cdot\frac{D}{M}\cdot W_s\cdot W_t,\chi};[\frac{rs^2t\alpha x}{DC}])\\
  &=\sum(-1)^{\nu(\alpha\alpha_t)}\cdot\chi(\alpha\alpha_t)\cdot a_0(E_{M,\ff_\chi\cdot\frac{D}{M}\cdot W_s,\chi};[\frac{rs^2t\alpha\alpha_t x}{DC}])\\
  &=\sum(-1)^{\nu(\alpha\alpha_t\alpha_s)}\cdot\chi(\alpha\alpha_t\alpha_s)\cdot\alpha^{-2}_s \cdot a_0(E_{M,\ff_\chi\cdot\frac{D}{M},\chi};[\frac{rs^2t\alpha\alpha_t\alpha_s x}{DC}]),
\end{align*}
where $\alpha_s,\alpha_t$ and $\alpha$ runs over all positive divisors of $W_s,W_t$ and $W$ respectively. Since
\begin{align*}
[\frac{rs^2t\alpha\alpha_t\alpha_s x}{DC}]=[\frac{r(s\alpha_t)^2(\frac{t\alpha}{\alpha_t})(\alpha_s x+\frac{DC}{\alpha^2_s})}{DC}]
\end{align*}
with $(\alpha_s x+\frac{DC}{\alpha^2_s},D)=1$ and $\alpha_s x+\frac{DC}{\alpha^2_s}\equiv \alpha_sx\pmod{\ff_\chi}$, we find by Lemma~\ref{constant2} that
\begin{align*}
  a_0(E_{M,\ff_\chi\cdot\frac{D}{M},\chi};[\frac{rs^2t\alpha\alpha_t\alpha_s x}{DC}])=(-1)^{\nu(\alpha_t)}\cdot\chi(\alpha\alpha_t\alpha_s)\cdot\alpha_t^{-1}\cdot a_0(E_{M,\ff_\chi\frac{D}{M},\chi};[\frac{rs^2tx}{DC}]),
\end{align*}
and hence
\[a_0(E_{M,L,\chi};[\frac{rs^2tx}{DC}])=\sum(-1)^{\nu(\alpha\alpha_s)}\cdot\alpha^{-1}_t\cdot\alpha^{-2}_s \cdot a_0(E_{M,\ff_\chi\cdot\frac{D}{M},\chi};[\frac{rs^2tx}{DC}]).\]
The above sum is zero unless $\frac{D}{M}\mid rs,(s,\ff_\chi)=1,\ff_\chi\mid t\text{ and }W=1$, or equivalently, $\frac{D}{M}\mid rs,(s,\ff_\chi)=1\text{ and }(M,L)\mid st$. If these conditions are satisfied, then we derive the desired result from the previous lemma.
\end{proof}

\begin{cor}\label{R}If $(M,L,\chi)\in\CH(DC)$ with $\chi$ a quadratic character of conductor $f_\chi$, then
\[\CR_{DC}(E_{M,L,\chi})=n_\chi\frac{\varphi({D}/{\ff_\chi})\psi({L}/{\ff_\chi})({D}/{M},C)}{L/\ff_\chi}\BZ.\]
\end{cor}
\begin{proof}
This follows immediately from the above result about constant terms, since the ramification index of $X_0(DC)$ at the cusp $[\frac{rs^2tx}{DC}]$ is $rs^2$.
\end{proof}

\subsection{}Now we turn to the calculation of the periods of the Eisenstein series $E_{M,L,\chi}$ with $\chi$ being a quadratic character.
\begin{lem}\label{Fourier'}
For any quadratic character $\chi$ with conductor $\ff_\chi$ dividing $C$, the Fourier expansion of $E_{D,\ff_\chi,\chi}$ at $[\infty]$ is
\begin{align*}
  E_{D,\ff_\chi,\chi}=a_0(E_{D,\ff_\chi,\chi};[\infty])+\sum^{\infty}_{n=1}\sigma_{{D}/{\ff_\chi}}(n)\cdot\chi(n)\cdot \fq^n,
\end{align*}
where, for any positive integer $n$, we have
\[\sigma_{{D}/{\ff_\chi}}(n)\colonequals\sum_{1\leq d\mid n,(d,{D}/{\ff_\chi})=1}d.\]

\end{lem}
\begin{proof}
We prove the statement by induction on $\nu({D}/{\ff_\chi})$. If $\nu(D/\ff_\chi)=1$ and hence $D=\ff_\chi$, then the assertion follows from (3.1) and (3.3) because $\chi^2=1$. Now suppose $\nu({D}/{\ff_\chi})>1$. If $p$ is a prime divisor of ${D/\ff_\chi}$, then we find by the induction that
\begin{align*}
  E_{D,\ff_\chi,\chi}&=[p]^+_\chi(E_{{D}/{p},\ff_\chi,\chi})\\
  &=\left(a_0(E_{{D}/{p},\ff_\chi,\chi})+\sum^{\infty}_{n=1}\sigma_{{D}/{\ff_\chi p}}(n)\cdot\chi(n)\cdot \fq^n\right)\\
  &-p\cdot\chi(p)\cdot\left(a_0(E_{{D}/{p},\ff_\chi,\chi})+\sum^{\infty}_{n=1}\sigma_{{D}/{\ff_\chi p}}(n)\cdot\chi(n)\cdot \fq^{pn}\right)\\
  &=a_0(E_{D,\ff_\chi,\chi})+\sum^{\infty}_{n=1}\left(\sigma_{{D}/{p\ff_\chi}}(n)-p\cdot\sigma_{{D}/{p\ff_\chi}}(n/p)\right)\cdot\chi(n)\cdot \fq^n,
\end{align*}
where ${n}/{p}$ is defined to be $0$ when $p\nmid n$. So the result follows because $\sigma_{{D}/{p\ff_\chi}}(n)-p\cdot\sigma_{{D}/{p\ff_\chi}}({n}/{p})=\sigma_{{D}/{\ff_\chi}}(n)$ for any positive integer $n$.
\end{proof}

\begin{lem}\label{Fourier}
For any $(M,L,\chi)\in\CH(DC)$ with $\chi$ a quadratic character of conductor $\ff_\chi$, we have
\[E_{M,L,\chi}=a_0(E_{M,L,\chi})+\sum^{\infty}_{n=1}\sigma_{M,L}(n)\cdot\chi(n)\cdot \fq^n,\]
where, for any positive integer $n$, we have
\begin{align*}
\sigma_{M,L}(n):=
\begin{cases}
(\prod_{\ell\mid{D}/{M}}\ell^{v_\ell(n)})\cdot\sigma_{D/\ff_\chi}(n)&\text{if}\ (n,{(M,L)}/{\ff_\chi})=1\\
0 &\text{otherwise}.
\end{cases}
\end{align*}

\end{lem}
\begin{proof}
First consider the case when $(M,L)=\ff_\chi$ and hence $E_{M,L,\chi}=E_{M,\ff_\chi\frac{D}{M},\chi}$. We proceed by induction on $\nu(D/M)$. When ${D}/{M}=1$ the result follows from Lemma~\ref{Fourier'}.
Now suppose ${D}/{M}>1$. If $p$ is a prime divisor of ${D}/{M}$, then we find by induction that
\begin{align*}
E_{M,\ff_\chi\frac{D}{M},\chi}&=[p]^-_{\chi}(E_{M,\ff_\chi\frac{D}{pM},\chi})\\
&=a_0(E_{M,\ff_\chi\frac{D}{M},\chi};[\infty])+\sum^{\infty}_{n=1}\left(\sigma_{M,\ff_\chi\frac{D}{pM}}(n)-\sigma_{M,\ff_\chi\frac{D}{pM}}(n/p)\right)\cdot\chi(n)\cdot \fq^n.
\end{align*}
Write $n=m\cdot p^{v_p(n)}$ with $(m,p)=1$, then
\begin{align*}
  &\sigma_{M,\ff_\chi\frac{D}{pM}}(n)-\sigma_{M,\ff_\chi\frac{D}{pM}}(n/p)\\
  &=(p^{v_p(n)}+...+1)\cdot\sigma_{M,\ff_\chi\frac{D}{pM}}(m)-(p^{v_p(n)-1}+...+1)\cdot\sigma_{M,\ff_\chi\frac{D}{pM}}(m)\\
  &=p^{v_p(n)}\cdot\sigma_{M,\ff_\chi\frac{D}{pM}}(n),
\end{align*}
which proves desired result.

Finally we complete the proof by induction on $\nu(\frac{(M,L)}{\ff_\chi})$. If $p$ is a prime divisor of $(M,L)/\ff_\chi$, then we find by induction that
\begin{align*}
  E_{M,L,\chi}&=[p]^-_{\chi}(E_{M,{L}/{p},\chi})\\
  &=a_0(E_{M,L,\chi};[\infty])+\sum^{\infty}_{n=1}\left(\sigma_{M,{L}/{p}}(n)-\sigma_{M,{L}/{p}}(n/p)\right)\cdot\chi(p)\cdot \fq^n,
\end{align*}
which proves the lemma because $\sigma_{M,{L}/{p}}(n)-\sigma_{M,{L}/{p}}(n/p)=0$ if $p\mid n$.
\end{proof}

\begin{prop}\label{P}
For any $(M,L,\chi)\in\CH(DC)$ with $\chi^2=1$, we have $\CP_{\Gamma_1(DC)}(E_{M,L,\chi})=\frac{g(\chi)}{L}\BZ+\CR_{\Gamma_1(DC)}(E_{M,L,\chi})$.
\end{prop}
\begin{proof}
Denote $\ff_\chi$ to be the conductor of $\chi$. For any character $\eta$ of conductor $\ff_\eta$ prime to $D$, it follows from Lemma~\ref{Fourier} that
\begin{align*}
L(E_{M,L,\chi},\eta,s)=\prod_{p\mid{M}/{\ff_\chi}}(1-\chi\eta(p)\cdot p^{1-s})\cdot\prod_{p\mid{L}/{\ff_\chi}}(1-\chi\eta(p)\cdot p^{-s})\cdot L(\chi\eta,s-1)\cdot L(\chi\eta,s)
\end{align*}
In particular, if $\chi\eta(-1)=1$, then $\Lambda(E_{M,L,\chi},\eta,1)=0$; on the other hand, if $\chi\eta(-1)=-1$, then
\begin{align*}
  \Lambda(E_{M,L,\chi},\eta,1)=
-\frac{\eta(-\ff_\chi)\chi(\ff_\eta)g(\chi)}{2\ff_\chi}\cdot\prod_{p\mid{M}/{\ff_\chi}}(1-\chi\eta(p))\cdot\prod_{p\mid{L}/{\ff_\chi}}(1-\frac{\chi\eta(p)}{p})\cdot B_{1,\chi\eta}\cdot B_{1,\overline{\chi\eta}}.
\end{align*}
Therefore $\frac{g(\chi)}{L}\BZ+\CR_{\Gamma_1(DC)}(E_{M,L,\chi})$ satisfies the condition (St3) by Theorem 4.2(b) of \cite{St2}, so that $\CP_{\Gamma_1(DC)}(E_{M,L})$ is contained in $\frac{g(\chi)}{L}\BZ+\CR_{\Gamma_1(DC)}(E_{M,L,\chi})$.

To complete the proof, it is thus sufficient to show that $\CP_{\Gamma_1(DC)}(E_{M,L,\chi})$ contains $\frac{g(\chi)}{L}\BZ$. Let $q$ be an arbitrary prime. Let $p'\in S_{DC}$ be a prime different from $q$. Then, for all but finitely many $\eta\in\fX^{\infty}_{DC}$ whose conductor is a power of $p'$, both $\prod_{p\mid {M}/{\ff_\chi}}(\chi(p)-\eta(p))$ and $\prod_{p\mid {L}/{\ff_\chi}}(\chi(p)\cdot p-\eta(p))$ are $q$-adic units. So it follows from Theorem 4.2(c) of \cite{St2} that $\frac{L}{g(\chi)}\cdot\Lambda(E_{M,L,\psi},\chi,1)$ is a $q$-adic unit for infinitely many $\eta\in\fX^{\infty}_{DC}$ and hence completes the proof.
\end{proof}

\begin{thm}\label{order} For any $(M,L,\chi)\in\CH(DC)$ with $\chi$ a quadratic character of conductor $\ff_\chi$, we have
\[\CC^{(DC)}_{M,L,\chi}\otimes_{\BZ}\BZ[{1}/{6}]\simeq\frac{\frac{g(\chi)}{\ff_\chi n_\chi}\BZ+\varphi(\frac{D}{\ff_\chi})\psi(\frac{L}{\ff_\chi})(\frac{D}{M},C)\BZ}{\varphi(\frac{D}{\ff_\chi})\psi(\frac{L}{\ff_\chi})(\frac{D}{M},C)\BZ}\otimes_{\BZ}\BZ[{1}/{6}] .\]
\end{thm}
\begin{proof}
Since $\CR_{DC}(E_{M,L,\chi})$ contains $\CR_{\Gamma_1(DC)}(E_{M,L,\chi})$, we find by Corollary~\ref{R} and Proposition~\ref{P} that
\begin{align*}
A^{(s)}(E_{M,L,\chi})&=\frac{\CP_{\Gamma_1(DC)}(E_{M,L,\chi})+\CR_{DC}(E_{M,L,\chi})}{\CR_{DC}(E_{M,L,\chi})}\\
&\simeq\frac{\frac{g(\chi)}{\ff_\chi n_\chi}\BZ+\varphi({D}/{\ff_\chi})\psi({L}/{\ff_\chi})(\frac{D}{M},C)\BZ}{\varphi({D}/{\ff_\chi})\psi({L}/{\ff_\chi})\cdot(\frac{D}{M},C)\BZ}
\end{align*}
Thus, to prove the theorem, it suffices to show that $\CC^{(DC)}_{M,L,\chi}\bigcap\sum_{DC}$ is annihilated by $6$.

$\bullet$ If $\chi=1$, then $\CC^{(DC)}_{M,L,\chi}\bigcap\Sigma_{DC}$ is $\BQ$-rational and of multiplicative type. So it must be contained in $\mu_2$ and hence annihilated by $2$.

$\bullet$ If $\chi$ is nontrivial so that $\ff_\chi>1$ is odd. By Proposition~\ref{eigen}, $T_p$ annihilates $\CC^{(DC)}_{M,L,\chi}$ for each prime divisor $p$ of $\ff_\chi$. On the other hand, by \cite{LO}, any such $T_p$ acts on $\sum_{DC}$ as multiplication by $p$. So we find that $\CC^{(DC)}_{M,L,\chi}\bigcap\sum_{DC}\subseteq\mu_{\ff_\chi}$. However, since $\chi$ can be cyclotomic only if $\ff_\chi=3$, we find that $\sum_{DC}\bigcap C^{(DC)}_{M,L,\chi}$ is zero away from $3$ and hence completes the proof.
\end{proof}

\begin{remark}
For any character $\chi$ of conductor $\ff_\chi$, let
\[d_\chi\colonequals\sum_{a,b\in\BZ/\ff_\chi\BZ}\chi(a)\chi(b)B_2(\frac{a+b}{\ff_\chi}),\]
which is clearly $6\ff_\chi$-integral. Since $g(\chi)^2=\pm \ff_\chi$, it follows that
\begin{align*}
\frac{g(\chi)}{\ff_\chi n_\chi}=-\frac{4g(\chi)^2}{\ff^2_\chi}\cdot \frac{1}{d_\chi}=\pm\frac{4}{\ff_\chi d_\chi},
\end{align*}
and we find from Proposition~\ref{order} that
\begin{align}
\CC^{(DC)}_{M,L,\chi}\otimes_{\BZ}\BZ[\frac{1}{6\ff_\chi}]\simeq \frac{\BZ[\frac{1}{6\ff_\chi}]}{\varphi(\frac{D}{\ff_\chi})\psi(\frac{L}{\ff_\chi})(\frac{D}{M},C)d_\chi \BZ[\frac{1}{6\ff_\chi}]}.
\end{align}
\end{remark}

\section{Proof of Theorem~\ref{M}}

\subsection{}We first recall some notations from the algebraic theory of modular forms. For more details of this theory, we refer the reader to \S1 of \cite{Oh}. For any positive integer $N$, let $M^B_2(\Gamma_0(N),\BZ)$ (resp. $S^B_2(\Gamma_0(N),\BZ)$) be the sub-$\BZ$-module of $\CM_2(\Gamma_0(N),\BC)$ (resp. $S_2(\Gamma_0(N),\BC)$) whose Fourier expansions at infinity have coefficients in $\BZ$. Then, for an arbitrary commutative ring $R$, we define
\begin{align*}
\begin{cases}
M^B_2(\Gamma_0(N),R)\colonequals M^B_2(\Gamma_0(N),\BZ)\otimes_{\BZ}R\\
S^B_2(\Gamma_0(N),R)\colonequals S^B_2(\Gamma_0(N),\BZ)\otimes_{\BZ}R.
\end{cases}
\end{align*}
If $N$ is invertible in $R$, then, by a moduli theoretic method, we can define $R$-modules
\[S^A_2(\Gamma_0(N),R)\subseteq M^A_2(\Gamma_0(N),R).\]
By (1.3.4) of \cite{Oh}, if $R$ is flat over $\BZ[1/N]$, then
\begin{align*}
\begin{cases}
  M^B_2(\Gamma_0(N),R)=M^A_2(\Gamma_0(N),R)\\
  S^B_2(\Gamma_0(N),R)=S^A_2(\Gamma_0(N),R).
\end{cases}
\end{align*}
Note that, for any $g\in C^{\infty}(\CH,\BC)$ and $z\in\CH$, we have
\begin{align*}
\begin{cases}
[p]^+_\chi(g)(z)=g(z)-p\cdot\chi(p)\cdot g(pz)\\
[p]^-_\chi(g)(z)=g(z)-\chi(p)^{-1}\cdot g(pz),
\end{cases}
\end{align*}
where $\chi$ is a Dirichlet character and $p$ a prime not dividing the conductor of $\chi$. So it follows from (3.1) and (3.2) that $E_{M,L,\chi}$ belongs to $M^B_2(\Gamma_0(DC),\BZ[1/6])$. In particular we find that
\begin{align*}
  E_{M,L,\chi}&\in M^B_2(\Gamma_0(DC),\BZ[1/6D])\\
  &=M^A_2(\Gamma_0(DC),\BZ[1/6D]).
\end{align*}

\subsection{}For any $(M,L,\chi)$ in $\CH(DC)$, denote by $I^{(DC)}_{M,L,\chi}$ the Eisenstein ideal $I_{(DC)}(E_{M,L,\chi})$. When $\chi$ is a \emph{quadratic} character, we call $I^{(DC)}_{M,L,\chi}$ a \emph{quadratic} Eisenstein ideal. Then, since $\chi^2=1$, all values of $\chi$ are either $1$ or $-1$, and it follows from Proposition~\ref{eigen} that $I^{(DC)}_{M,L,\chi}$ is generated by the following elements of $\BT(DC)$:
\begin{itemize}
  \item $T_\ell-\chi(\ell)-\chi(\ell)\cdot\ell$ for primes $\ell\nmid D$;
  \item $T_p-\chi(p)$ for primes $p\mid M/(M,L)$;
  \item $T_p-\chi(p)\cdot p$ for primes $p\mid L/(M,L)$;
  \item $T_p$ for primes $p\mid(M,L)$.
\end{itemize}

\begin{lem}\label{5.1}Notations are as above. Then:
\begin{enumerate}
  \item $\BT(DC)/I^{(DC)}_{M,L,\chi}$ is a finite cyclic group;
  \item There is a surjection $\BT(DC)/I^{(DC)}_{M,L,\chi}\twoheadrightarrow  \CC^{(DC)}_{M,L,\chi}$.
\end{enumerate}
\end{lem}
\begin{proof}
Since $T_\ell$ is congruent to an integer modulo $I^{(DC)}_{M,l,\chi}$ for any prime $\ell$, there is a surjection $\phi:\ \BZ\rightarrow\BT(DC)/I^{(DC)}_{M,L,\chi}$. Thus, to prove (1), we only need to show $\ker(\phi)$ is nonzero. Suppose to the contrary that $\ker(\phi)=0$ so that $\phi$ is an isomorphism, then we obtain a normalized eigenform $f=\sum^{\infty}_{n=1}a_n(f)\fq^n$ with $a_\ell(f)=\chi(\ell)(1+\ell)$ for any prime $\ell\nmid D$. But this contradicts to the Ramamujan bound, so we fine that $\ker(\phi)$ must be nonzero.

By Theorem 3.2.4 of \cite{St}, $\CC^{(DC)}_{M,L,\chi}$ is a cyclic $\BT(DC)$-module. Since $I^{(DC)}_{M,L,\chi}$ annihilates $\CC^{(DC)}_{M,L,\chi}$ (see Remark~\ref{rmk}), the action of $\BT(DC)$ on $\CC^{(DC)}_{M,L,\chi}$ induces the desired surjection,
\end{proof}

\begin{thm}\label{index}There is an isomorphism
\[\BT(DC)/I^{(DC)}_{M,L,\chi}\otimes_{\BZ}\BZ[1/6D]\simeq\CC^{(DC)}_{M,L,\chi}\otimes_{\BZ}\BZ[1/6D].\]
\end{thm}
\begin{proof}
Let $p$ be a prime not dividing $6D$. By Lemma~\ref{5.1}(1),  we have $\BT(DC)_{p}/I^{(DC)}_{M,L,\chi}\simeq\BZ/p^{m}\BZ$ for some integer $m\geq0$. By Theorem 2.2 of \cite{Ri}, there is a perfect pairing of $\BZ$-modules
\[\BT(DC)\times S^B_2(\Gamma_0(DC),\BZ)\rightarrow\BZ,\ (T,g)\mapsto a_1(T(g);[\infty]).\]
Let $S^B_2(\Gamma_0(DC),\BZ/p^{m}\BZ)[I^{(DC)}_{M,L,\chi}]$ be the submodule of $S^B_2(\Gamma_0(DC),\BZ/p^{m}\BZ)$ annihilated by $I^{(DC)}_{M,L,\chi}$. It follows that we have an induced perfect pairing of $\BZ/p^{m}\BZ$-modules
\[(\BT(DC)/I^{(DC)}_{M,L,\chi})_{p}\times S^B_2(\Gamma_0(DC),\BZ/p^{m}\BZ)[I^{(DC)}_{M,L,\chi}]\rightarrow\BZ/p^{m}\BZ,\]
and hence $S^B_2(\Gamma_0(DC),\BZ/p^{m}\BZ)[I^{(DC)}_{M,L,\chi}]=({\BZ}/{p^{m}\BZ})\cdot\theta$ for some normalized eigenform $\theta$. On the other hand, since $E_{M,L,\chi}\pmod{p^{m}\BZ}$ belongs to $M^B_2(\Gamma_0(DC),\BZ/p^{m}\BZ)$ and has the same Hecke eigenvalues as those of $\theta$, we obtain the following constant form
\begin{align*}
  A&\colonequals E_{M,L,\chi}\pmod{p^{m}\BZ}-\theta\\
  &=a_0(E_{M,L,\chi};[\infty])\pmod{p^m}
\end{align*}
Below we distinguish into two situations:
\begin{itemize}
  \item If $L\neq1$ so that $a_0(E_{M,L,\chi};[\infty])=0$ by Proposition~\ref{constant3}, then $A=0$ and hence $E_{M,L,\chi}\pmod{p^{m}\BZ}=\theta$ by the $\fq$-expansion principle (see Proposition 1.2.10 of \cite{Oh}). Thus we find that $E_{M,L,\chi}\pmod{p^{m}\BZ}$ belongs to $S^B_2(\Gamma_0(DC),\BZ/p^m\BZ)$. Since $p$ is prime to $D$, it follows from Lemma (1.3.5) of \cite{Oh} that $E_{M,L,\chi}\pmod{p^{m}\BZ}$ belongs to $S^A_2(\Gamma_0(DC),\BZ/p^m\BZ)$, and hence vanishes at all cusps. By Proposition~\ref{constant3}, the constant term at any cusp with $rs=\frac{D}{M},(s,\ff_\chi)=1$ and $(M,L)\mid t$ (for example at the cusp $[\frac{(M,L)}{M(M,C)}]$) is of the form
\[u\cdot\varphi(\frac{D}{\ff_\chi})\cdot\psi(\frac{L}{\ff_\chi})\cdot d_\chi,\]
where $u$ is a $p$-adic unit. It follows that $p^{m}$ divides $\varphi(\frac{D}{\ff_\chi})\cdot\psi(\frac{L}{\ff_\chi})\cdot d_\chi$, which proves the desired assertion in view of (4.7).
  \item If $L=1$, then $E_{M,L,\chi}=E_{D,1}$ whose constant term at infinity is $\pm\frac{1}{24}\varphi(D)$. Let $q$ be an auxiliary prime with $q\nmid D$ and $q\neq\pm1\pmod{p}$. Let $B(q)$ be the operator introduced on P289 of \cite{Oh}, which equals $\frac{1}{q}\gamma_q$ when base change to $\BC$. Then
\begin{align*}
0&=\left(1-q\cdot B(q)\right)(A)\\
&=E_{D,q}\pmod{p^m\BZ}-\left(1-q\cdot B(q)\right)(\theta),
\end{align*}
which implies that $E_{D,q}\pmod{p^m\BZ}$ is a cuspform. In particular $a_0(E_{D,q};[1])\equiv0$ is congruent to zero modulo ${p^m}$, which implies by Proposition~\ref{constant3} that $p^m$ divides $\varphi(Dq)\cdot\psi(q)=\varphi(D)\cdot(q^2-1)$. It follows that $p^m$ divides $\varphi(D)$ as desired in view of (4.7) and hence completes the proof.
\end{itemize}
\end{proof}

\subsection{}For any prime divisor $p$ of $D$, there are two degeneracy maps
\[\pi^{(p)}_1,\ \pi^{(p)}_p:X_0(DC)\rightarrow X_0(DC/p),\]
which can be analytically described as $\pi^{(p)}_1(z)=z$ and $\pi^{(p)}_p(z)=pz$ for any $z\in\CH$. Thus, from the Picard functoriality, we obtain a homomorphism
\[\iota_p=\pi^{(p)*}_1+\pi^{(p)*}_p:J_0(DC/p)^2\rightarrow J_0(DC)\]
between abelian varieties over $\BQ$. Define
\begin{align*}
\begin{cases}
J_0(DC)_{p\text{-old}}\colonequals\Im(\iota_p)\\
J_0(DC)^{p\text{-new}}\colonequals J_0(DC)/J_0(DC)_{p\text{-old}}.
\end{cases}
\end{align*}
Both $J_0(DC)_{p\text{-old}}$ and $J_0(DC)^{p\text{-new}}$ are stable under the action of $\BT_0(DC)$, so we can define
\begin{align*}
\begin{cases}
\BT(DC)^{p\text{-old}}:=\Im(\BT_0(DC)\rightarrow \End_{\BQ}(J_0(DC)_{p\text{-old}}))\\
\BT(DC)^{p\text{-new}}:=\Im(\BT_0(DC)\rightarrow \End_{\BQ}(J_0(DC)^{p\text{-new}})).
\end{cases}
\end{align*}
We use the same symbols for the Hecke operators in $\BT(DC)$ and their images in $\BT(DC)^{p\text{-old}}$ or $\BT(DC)^{p\text{-new}}$ for simplicity. It follows that there are two surjective $\BZ$-algebra homomorphisms
\begin{align}
\begin{cases}
\BT(DC)\twoheadrightarrow\BT(DC)^{p\text{-old}}\\
\BT(DC)\twoheadrightarrow\BT(DC)^{p\text{-new}},
\end{cases}
\end{align}
which combine to give the following homomorphism with nilpotent kernel
\begin{align}
\BT(DC)\rightarrow\BT(DC)^{p\text{-old}}\times\BT(DC)^{p\text{-new}}.
\end{align}
There is a natural inclusion of $\BT(DC/p)$ in $\End_{\BQ}(J_0(DC)_{p\text{-old}})$ induced from the diagonal action of $\BT(DC/p)$ on $J(DC/p)^2$. Put $R:=\BT(DC/p)\bigcap\BT(DC)^{p\text{-old}}$. For any prime $\ell\neq p$ and $i=1\text{ or }p$, we have
\[ T^{(DC)}_\ell\circ\pi^{(p)*}_i=\pi^{(p)*}_i\circ T^{(DC/p)}_\ell,\]
so $R$ is generated by all $\ell$-th Hecke operator for primes $\ell$ different from $p$. In the following discussion we distinguish into two situations:

$\bullet$ If $p$ is a prime divisor of $D/C$ so that $p$ exactly divide $DC$, then by the lemma on P491 of \cite{Wiles} we have
\begin{align}
R\otimes_{\BZ}\BZ[1/2]=\BT(DC/p)\otimes_{\BZ}\BZ[1/2].
\end{align}
Moreover, if $p$ is odd, then (loc.cit.)
\[R=\BT(DC/p),\]
and there is an isomorphism (see \S7 of \cite{Ri2})
\begin{align}
\BT(DC/p)[x]/(x^2-T^{(DC/p)}_p x+p)\simeq\BT(DC)^{p\text{-old}},
\end{align}
which maps $T(DC/p)_{\ell}$ to $T^{(DC)}_\ell$ for any prime $\ell\neq p$ and maps $x$ to $T^{(DC)}_p.$

$\bullet$ On the other hand, if $p$ is a prime divisor of $C$ so that $p^2\mid DC$, then straightforward verification shows that
\begin{align*}
\begin{cases}
T^{(DC)}_p=\pi^{(p)*}_p\circ \pi^{(p)}_{1*}\\
T^{(DC/p)}_p=\pi^{(p)}_{1*}\circ \pi^{(p)*}_{p},
\end{cases}
\end{align*}
which implies
\begin{align}
\begin{cases}
T^{(DC)}_p\circ\pi^{(p)*}_p=\pi^{(p)*}_p\circ T^{(DC/p)}_p\\
T^{(DC)}_p\circ\pi^{(p)*}_1=p\cdot\pi^{(p)*}_{p}.
\end{cases}
\end{align}
Therefore $\pi^{(p)*}_p(J_0(DC/p))$ is stable under the action of $\BT(DC)^{p\text{-old}}$, and we obtain an induced surjection
\[\varphi'_p:\BT(DC)^{p\text{-old}}\twoheadrightarrow\BT(DC/p),\]
which maps $T^{(DC)}_\ell$ to $T^{(DC/p)}_\ell$ for \emph{any} prime $\ell$. Moreover, denote
\begin{align*}
  J_0(DC)^{(1)}_{p\text{-new}}\colonequals J_0(DC)_{p\text{-old}}/\pi^{(p)*}_p(J_0(DC/p)),
\end{align*}
which is also stable under $\BT(DC)^{p\text{-old}}$, and define
\begin{align*}
  \BT(DC)^{p\text{-new}}_{(1)}\colonequals\Im(\BT(DC)^{p\text{-old}}\rightarrow\End_{\BQ}(J_0(DC)^{(1)}_{p\text{-new}})).
\end{align*}
Then we obtain a surjection
\begin{align*}
  \varphi''_p:\BT_0(DC)^\text{p-old}\twoheadrightarrow{\BT(DC)^{p\text{-new}}_{(1)}},
\end{align*}
which maps $T^{(DC)}_p$ to zero by (6.5). Note that $\pi^{(p)*}_1$ induces a surjection from $J_0(DC/p)\rightarrow J_0(DC)^{p\text{-new}}_{(1)}$, it follows that $\BT(DC)^{p\text{-new}}_{(1)}$ is a quotient ring of $\BT(DC/p)$. Finally, the two homomorphisms $\varphi'_p$ and $\varphi''_p$ combine to give the following homomorphism with nilpotent kernel
\begin{align}
\varphi_p:\BT(DC)^{p\text{-old}}\hookrightarrow\BT(DC/p)\times{\BT(DC)^{p\text{-new}}_{(1)}}.
\end{align}

\subsection{}Form now on let $D$ be an odd square-free positive integer and $C$ a divisor of $D$ . Let $\chi$ be a quadratic character with conductor $\ff_\chi$ dividing $C$. Define
\[I_{\chi}\colonequals\left(T^{(DC)}_\ell-\chi(\ell)-\chi(\ell)\cdot\ell\right)_{\ell\nmid D}.\]

\begin{lem}\label{6.1}$\BT(DC)/I_\chi$ is a finite ring.
\end{lem}
\begin{proof}
Since $\BT(DC)$ is finite over $\BZ$, $\BT(DC)/I_\chi$ is finitely generated over $\BZ$. Thus, to prove the lemma, it suffices to show that $(\BT(DC)/I_\chi)\otimes_{\BZ}\BC=0$.

Suppose to the contrary that $(\BT(DC)/I_\chi)\otimes_{\BZ}\BC$ is nonzero. Then by the Hilbert Nullstellensatz, there exists a nonzero $\BZ$-algebra homomorphism from $\BT(DC)/I_\chi$ to $\BC$, which in turn gives a normalized cuspidal eigenform whose $\ell$-th Hecke eigenvalue is $\chi(\ell)+\chi(\ell)\cdot\ell$ for any prime $\ell\nmid D$. But this contradicts to the Ramanujan bound, so $(\BT(DC)/I_\chi)\otimes_{\BZ}\BC$ must be zero.
\end{proof}

In particular $\BT(DC)/I_\chi$ is an artinian ring. Therefore, for any prime $q$, we have
\begin{align}\label{decomposition}
\BT(DC)_q/I_\chi\simeq\prod_{\fm}\BT(DC)_\fm/I_\chi,
\end{align}
where $\fm$ runs over all the maximal ideals containing the ideal $(q,I_\chi)$.

\begin{prop}\label{prop6.1}If $q$ is a prime with $(q,2D)=1$ and $\fm$ a maximal ideal such that $\fm\supseteq(q,I_\chi)$, then
\begin{align*}
T^{(DC)}_\ell\pmod{\fm}\equiv
\begin{cases}
\chi(\ell)+\chi(\ell)\cdot\ell&\text{if }\ell\nmid D\\
\chi(\ell)\text{ or }\chi(\ell)\ell&\text{if }\ell\mid{D}/{C}\\
0,\ \chi(\ell)\text{ or }\chi(\ell)\ell&\text{if }\ell\mid C.
\end{cases}
\end{align*}
Moreover, if $p$ is a prime divisor of $\ff_\chi$, then $T^{(DC)}_p\equiv0\pmod{\fm}$ and $\BT(DC)^{p\text{-old}}_{\fm}=0$.
\end{prop}
\begin{proof}
For simplicity, we will also denote by $\fm$ to be its images in various old- or new-quotients of $\BT(DC)$. Let $\rho_\fm:\Gal(\overline{\BQ}/\BQ)\rightarrow\GL_2(\BT(DC)/\fm)$ be the unique semi-simple Galois representation associated to $\fm$ (see \S5 of \cite{Ri2}). For any prime $\ell$ not dividing $D$, the assertion follows directly from the definition of $I_\chi$. Therefore, by Chebotarev density theorem and the theorem of Brauer-Nesbitt, we have $\rho_\fm\simeq\chi\oplus\chi\epsilon_q$, where $\epsilon_q$ is the modulo $q$ cyclotomic character.

\begin{itemize}
  \item Let $p$ be a prime divisor of $D/C$. If $\BT(DC)^{p\text{-new}}_{\fm}$ is nonzero and hence $\BT(DC)/\fm\simeq\BT(DC)^{p\text{-new}}/\fm$, then $\rho_{\fm}$ comes from a $p$-new form. Then it follows from Theorem 3.1,(e) of \cite{DDT} that $\rho_\fm\simeq\eta\oplus\eta\epsilon_q$ when restricted to $\Gal(\overline{\BQ}_p/\BQ_p)$, and $\eta(\Frob_p)=T^{(DC)}_p\pmod{\fm}$. So we find that $T^{(DC)}_p\equiv\chi(p)\text{ or }p\chi(p)\pmod{\fm}$.
  \item Still assume that $p$ is a prime divisor of $D/C$. If $\BT(DC)^{p\text{-new}}_\fm=0$, then $\BT(DC)/\fm\simeq\BT(DC)^{p\text{-old}}/\fm$. Since $q$ is odd, we have $\BT(DC/p)_q[x]/(x^2-T^{(DC/p)}_p x+p)\simeq\BT(DC)^{p\text{-old}}_q$ by (5.3) and (5.4). Let $\fn$ be the inverse image of $\fm$ in $\BT(DC/p)$. Then $\fn$ contains the element $T^{(DC/p)}_\ell-\chi(\ell)-\chi(\ell)\ell$ for any prime $\ell\nmid D$. In particular $\rho_\fn\simeq\rho_\fm$ by Chebotarev density theorem and the theorem of Brauer-Nesbitt. Now since $q\neq p$ and $p\nmid DC/p$, it follows by Proposition 5.1 of \cite{Ri2} that $T^{(DC/p)}\equiv\chi(p)+\chi(p) p\pmod{\fn}$. Therefore $T^{(DC)}_p\equiv\chi(p)$ or $\chi(p)\cdot p\pmod{\fm}$ as desired.
  \item Let $p$ be a prime divisor of $C$. If $\BT(DC)^{p\text{-new}}_{\fm}$ is nonzero and hence $\BT(DC)/\fm\simeq\BT(DC)^{p\text{-new}}/\fm$, then it follows from the newform theory that $T^{(DC)}_p\equiv0\pmod{\fm}$ as $p^2$ divides $DC$.
  \item Still assume that $p$ is a prime divisor of $C$. If $\BT(DC)^{p\text{-new}}_\fm=0$, then $\BT(DC)/\fm\simeq\BT(DC)^{p\text{-old}}/\fm$, and hence one of $\varphi'_p(\fm)$ and $\varphi''_p(\fm)$ must be maximal. In the first case, there is an induced isomorphism $\BT(DC)^{p\text{-old}}/\fm\simeq\BT(DC/p)/\fm$, so we are done by the above already proved result. In the second case, we have $\BT(DC)^{p\text{-old}}/\fm\simeq\BT(DC)^{p\text{-new}}/\fm$, which implies that $T^{(DC)}_p\equiv0\pmod{\fm}$ because $\varphi'_p$ maps $T^{(DC)}_p$ to zero.
  \item Finally, let $p$ be a prime divisor of $\ff_\chi$. Suppose $\BT(DC)^{p\text{-old}}/\fm$ is nonzero, then we find as above that one of $\varphi'_p(\fm)$ and $\varphi''_p(\fm)$ must be maximal. Since $\varphi'_p(\fm)\subseteq\BT(DC/p)$, $\varphi''_p(\fm)\subseteq\BT(DC)^{p\text{-new}}_{(1)}$, and $\BT(DC)^{p\text{-new}}_{(1)}$ is a quotient ring of $\BT(DC/p)$ as we remarked before, we will always obtain some maximal ideal $\fn$ in $\BT(DC/p)$ which contains $T^{(DC/p)}_{\ell}-\chi(\ell)-\ell\chi(\ell)$ for all those primes $\ell$ not dividing $D$. In particular $\rho_{\fn}\simeq\chi\oplus\chi\epsilon_q$, which is ramified at $p$.

      However, since $p$ exactly divides $(D/C)$, the restriction of $\rho_\fn$ to ${\Gal(\overline{\BQ}_p/\BQ_p)}$ has an unramified quotient by Theorem 3.1,(e) of \cite{DDT}.  This is a contradiction. So we have $\BT(DC)^{p\text{-old}}_{\fm}=0$ and hence  $\BT(DC)/{\fm}\simeq\BT(DC)^{p\text{-new}}/{\fm}$, which implies $T^{(DC)}_p\equiv0\pmod{\fm}$ by the newform theory.
\end{itemize}
\end{proof}

\subsection{}Let $q$ be a prime with $(q,2D)=1$. Fix a maximal ideal $\fm$ containing $(q,I_\chi)$. Let
\begin{align*}
  \CP_1(\fm)&=\{p\text{ a prime}: p\mid D,\  T^{(DC)}_p\equiv0,\chi(p)\pmod{\fm}\}\\
  \CP_2(\fm)&=\{p\text{ a prime}: p\mid D,\  T^{(DC)}_p\equiv0,\chi(p)p\pmod{\fm},\text{ and }p\neq1\pmod{q}\},
\end{align*}
and we define
\begin{align*}
\begin{cases}
  M=\prod_{p\in\CP_1(\fm)}p\\
  L=\prod_{p\in\CP_2(\fm)}p.
\end{cases}
\end{align*}

\begin{lem}\label{construct Eisenstein}Notations as above. Then $(M,L,\chi)$ belongs to $\CH(DC)$.
\end{lem}
\begin{proof}It is clear that $\fm=(q,I^{(DC)}_{M,L,\chi})$. To prove $(M,L,\chi)\in\CH(DC)$, we only need to show that $M>1$. Suppose $M=1$. Then $\chi$ is is the trivial character. Moreover, for any prime divisor $p$ of $D$, we have $T^{(DC)}_p\equiv p\pmod{q}$ and $p\neq1\pmod{q}$.

If $p$ is a prime divisor of $C$, then $T^{(DC)}_p$ is mapped to zero in $\BT(DC)^{p\text{-new}}$, so the image of $\fm$ in $\BT(DC)^{p\text{-new}}$ is the unit ideal because it contains $T^{(DC)}_p-p=-p$. In  particular $\BT(DC)^{p\text{-new}}/\fm=0$ and hence $\BT(DC)^{p\text{-old}}/\fm$ must be nonzero. Similarly, since $T^{(DC)}_p$ is mapped to zero via $\varphi''_p$, it follows that $\BT(DC)^{p\text{-new}}/\fm=0$. Therefore $\fm$ must be a maximal ideal in $\BT(DC/p)$. Then we find inductively that there exists a maximal ideal $\bar{\fm}$ in the Hecke algebra $\BT(D)$ such that: $T^{({D})}_\ell\equiv1+\ell\pmod{\bar{\fm}}$ for any prime $\ell$ not dividing ${D}$, $T^{({D})}_p\equiv p\pmod{\bar{\fm}}$ for any prime divisor $p$ of $D$.

If $p'$ is a prime divisor of $D$ and $\BT(D)^{p'\text{-old}}/\bar{\fm}$ is nonzero, then
\begin{align*}
  \BT(D)/\bar{\fm}&\simeq\BT(D)^{{p'}\text{-old}}/\bar{\fm}\\
  &\simeq\BT(D/p')[x]/(x^2-T^{(D/p')}_{p'}x+p',\ \bar{\fm}),
\end{align*}
which implies that there is a maximal ideal ${\fm}'$ in $\BT(D/p)$ such that: $T^{({D/p})}_\ell\equiv1+\ell\pmod{{\fm'}}$ for any prime $\ell$ not dividing ${D}$, $T^{({D/p'})}_{p}\equiv p\pmod{{\fm'}}$ for any prime divisor $p$ of $D/p'$. Proceeding in this way, we obtain a divisor $d$ of $D$ and a maximal ideal ${\fn}$ in $\BT(d)$ such that: $T^{({d})}_\ell\equiv1+\ell\pmod{{\fn}}$ for any prime $\ell$ not dividing ${d}$, $T^{({d})}_p\equiv p\pmod{{\fn}}$ and $\fn$ is $p$-new for any prime divisor $p$ of $d$. But this contradicts to Theorem 2.6,(ii) of \cite{BD} since $p\neq1\pmod{q}$ for any prime $p$ dividing $d$, so we complete the proof.
\end{proof}

\begin{remark}
  By the above lemma, to any maximal ideal $\fm$, we have an associated Eisenstein series $E_{M,L,\chi}$ such that $\fm=(q,I^{(DC)}_{M,L,\chi})$.
\end{remark}

\subsection{Proof of Theorem~\ref{M}:}Let $q$ be a prime such that $(q,6D\cdot\varphi(D/C)\cdot\varpi(C/\ff_\chi))=1$. By the Eichler-Shimura relation, $J_0(DC)(\chi)_{\tor}$, as an \etale group scheme over $\BZ[1/D]$, is annihilated by $I_\chi$-module. Therefore $J_0(DC)(\chi)[q^\infty]$ is a $\BT(DC)_q/I_{\chi}$-module, and it follows from \ref{decomposition} that
\begin{align*}
  \begin{cases}
    J_0(DC)(\chi)[q^{\infty}]=\bigoplus_{\fm} J_0(DC)(\chi)[{\fm}^{\infty}]\\
    C_0(DC)(\chi)[q^{\infty}]=\bigoplus_{\fm} C_0(DC)(\chi)[{\fm}^{\infty}],
  \end{cases}
\end{align*}
where $\fm$ runs over all maximal ideals containing $(q,I_\chi)$.

Fix such a maximal ideal $\fm$. Let $(M,L,\chi)$ be the associated triple as in Lemma~\ref{construct Eisenstein}. We claim that $J_0(DC)(\chi)[\fm^{\infty}]$ is annihilated by $I^{(DC)}_{M,L,\chi}$. Let $p$ be a prime divisor of $D$. Then there is an exact sequence of $\BT(DC)$-modules
   \[
\xymatrix@C=0.5cm{
  0 \ar[r] & J_0(DC)_{p\text{-old}}(\chi)[\fm^\infty] \ar[r] & J_0(DC)(\chi)[\fm^\infty] \ar[r] & J_0(DC)^{p\text{-new}}(\chi)[\fm^\infty]}.
  \]
\begin{itemize}
  \item If $p$ divides $\ff_\chi$, then $J_0(DC)_{p\text{-old}}(\chi)[\fm^\infty]=0$ by Proposition~\ref{prop6.1}. It follows that $J_0(DC)(\chi)[\fm^\infty]$ is contained in $J_0(DC)^{p\text{-new}}(\chi)[\fm^\infty]$, and is therefore annihilated by $T^{(DC)}_p$.
  \item We next prove $\BT(DC)^{p\text{-new}}_{\fm}=0$ if $p$ divides $D/\ff_\chi$. For otherwise there would exist a $p$-newform with an associated $q$-adic representation $\rho$ such that $\bar{\rho}^{ss}=\rho_\fm\simeq\chi\oplus\chi\epsilon_q$, where $\bar{\rho}^{ss}$ is the semisiplification of the reduction $\bar{\rho}$ of $\rho$ modulo $q$. So we are in the so-called degenerate case (see \cite{Ca} and \cite{Li}). However, when $p$ divides $C$, this degeneration is impossible since $(q,p^2-1)=1$ (see Proposition 1.1 of \cite{BD}). When $p$ divides $D/C$, then the restriction of $\bar{\rho}$ to ${I_p}$ is trivial as $(q,p-1)=1$, and hence $\codim(\bar{\rho}^{I_p})=0$. Since the Swan conductor does not change under reduction (see Proposition 1.1 of \cite{Li}), the degeneration of $\rho$ at $p$ would imply that $\codim(\rho^{I_p})=1$, so we find that
      \begin{align*}
        \rho|_{I_p}\sim\left(
                         \begin{array}{cc}
                           1 & b \\
                           0 & 1 \\
                         \end{array}
                       \right),
      \end{align*}
      where $b$ a non-trivial homomorphism from $I_p$ to a pro-$q$ group. But this is impossible again because $(q,p-1)=1$. So we prove the assertion.
  \item It follows that $J_0(DC)(\chi)[\fm^\infty]=J_0(DC)_{p\text{-old}}(\chi)[\fm^\infty]$ for any prime divisor $p$ of $D/\ff_\chi$. Since $J_0(DC)_{p\text{-old}}$ is a quotient of $J_0(DC/p)^2$, $T^{(DC/p)}_p$ acts on $J_0(DC)(\chi)[\fm^\infty]$ as multiplication by $\chi(p)+\chi(p)p$ by the Eichler-Shimura relation. It follows that $(T^{(DC)}_p-\chi(p))(T^{(DC)}_p-\chi(p)p)$ annihilates $J_0(DC)(\chi)[\fm^\infty]$. As $p\neq1\pmod{q}$, we find that $J_0(DC)(\chi)[\fm^\infty]$ is either annihilated by $T^{(DC)}_p-\chi(p)$ when $p\mid M$, or annihilated by
      $T^{(DC)}_p-\chi(p)p$ when $p\mid L$, which completes the proof of the claim.
\end{itemize}
In particular we find that $J_0(DC)(\chi)[\fm^\infty]$ is a $\BT(DC)_q/I^{(DC)}_{M,L,\chi}$-module. Consider $J_0(DC)[\fm]$ as a finite flat group scheme over $\BZ_q$. Note that $J_0(DC)(\chi)[\fm]$ is contained in $ J_0(DC)[\fm]^\et$. By composing $J_0(DC)(\chi)\hookrightarrow J_0(DC)(\overline{\BF}_q)$ with $J_0(DC)(\overline{\BF}_q)\hookrightarrow H^0(X_0(DC)_{/\overline{\BF}_q},\Omega)$ (see Proposition 14.7 of \cite{M}), we find that there is an injection $J_0(DC)(\chi)[\fm]\hookrightarrow H^0(X_0(DC)_{/\overline{\BF}_q},\Omega)[\fm]$. Since $H^0(X_0(DC)_{/\overline{\BF}_q},\Omega)[\fm]$ is isomorphic to $S^{B}(\Gamma_0(DC),\overline{\BF}_q)[\fm]$, it follows from the $\fq$-expansion principle that $\dim_{\BF_q}(J_0(DC)(\chi)[\fm])\leq1$ and so $J_0(DC)(\chi)[\fm^\infty]$ is a cyclic group. Therefore
\begin{align*}
|J_0(DC)(\chi)[\fm^\infty]|&\leq|\BT(DC)_q/I^{(DC)}_{M,L,\chi}|\\
&=|\CC^{(DC)}_{M,L,\chi}|,
\end{align*}
with the second equality holds by Proposition~\ref{index}. It is clear that $\CC^{(DC)}_{M,L,\chi}\subseteq J_0(DC)(\chi_{C})[\fm^\infty]$. So we find by the above inequality that $J_0(DC)(\chi)[\fm^\infty]=\CC^{(DC)}_{M,L,\chi}[q^\infty]$, which is contained in $C_0(DC)(\chi)[\fm^\infty]$, and hence complete the proof.

\end{document}